\documentclass[12pt]{article}

\usepackage{amssymb,amsmath,latexsym,color,graphicx,stmaryrd}

\usepackage{graphics}
\usepackage{epsfig}
\usepackage{color}
\usepackage{mathrsfs}

\usepackage{xcolor}

\definecolor{amethyst}{rgb}{0.6, 0.4, 0.8}

\def\wrtext#1{\relax\ifmmode{\leavevmode\hbox{#1}}\else{#1}\fi}
\def\abs#1{\left|#1\right|}
\def\begeq{\begin{equation}}
\def\endeq{\end{equation}}

\def\part#1{\frac{\partial}{\partial #1}}

\def\norm#1{||\,#1\,||}

\newcommand{\real}{\mbox{\bf R}}
\newcommand{\comp}{\mbox{\bf C}}

\newcommand{\nat}{\mbox{\bf N}}

\renewcommand{\Im}{\mbox{\rm Im\,}}
\renewcommand{\exp}{\mbox{\rm exp\,}}

\newtheorem{dref}{Definition}[section]
\newtheorem{lemma}[dref]{Lemma}
\newtheorem{theo}[dref]{Theorem}
\newtheorem{prop}[dref]{Proposition}

\newtheorem{coro}[dref]{Corollary}

\newenvironment{proof}{\vspace{.3cm}\noindent{{\em Proof:}}}{\hfill$\Box$}

\title{Semiclassical Gevrey operators in the complex domain}
\author{Michael \textsc{Hitrik} \footnote{Department of Mathematics, University of California, Los Angeles CA 90095-1555, USA, {\sf hitrik@math.ucla.edu}} \and Richard \textsc{Lascar} \footnote{JAD - UMR 7351, Universit\'e C\^ote d'Azur Parc Valrose 06108 Nice Cedex 02, France, {\sf richard.lascar@univ-cotedazur.fr}} \and Johannes \textsc{Sj\"ostrand}\footnote{IMB, Universit\'e de Bourgogne 9, Av. A. Savary, BP 47870
FR-21078 Dijon, France and UMR 5584 CNRS, {\sf johannes.sjostrand@u-bourgogne.fr}} \and Maher \textsc{Zerzeri} \footnote{LAGA - UMR7539 CNRS, Universit\'e Sorbonne Paris-Nord, 99, avenue J.-B. Cl\'ement F-93430 Villetaneuse, France, {\sf zerzeri@math.univ-paris13.fr}}}

\date{}

\begin{document}

\maketitle


\vspace*{0.5cm}
\noindent
{\bf Abstract}: We study semiclassical Gevrey pseudo\-dif\-ferential operators, acting on exponentially weighted spaces of entire holomorphic functions. The symbols of such operators are Gevrey functions defined on suitable I-Lagrangian submanifolds of the complexified phase space, which are extended almost holomorphically in the same Gevrey class, or in some larger space, to complex neighborhoods of these submanifolds. Using  almost holomorphic extensions, we obtain uniformly bounded realizations of such operators on a natural scale of exponentially weighted spaces of holomorphic functions for all Gevrey indices, with remainders that are optimally small, provided that the Gevrey index is $\leq 2$.



\tableofcontents

\section{Introduction and statement of results}\label{int}
\setcounter{equation}{0}
Starting with the pioneering work~\cite{BoKr67}, the study of (pseudo)differential operators with Gevrey coefficients has had a long tradition in the PDE theory, see also~\cite{Rodino},~\cite{La88}, \cite{LaLa91}, \cite{LaLa97}. The work~\cite{La88}, in particular, develops the semiclassical Weyl calculus of pseudodifferential operators on $\real^n$, with symbols having Gevrey regularity in both the position and momentum variables. Now it is well known that the Weyl quantization assumes a particularly simple and convenient form when passing from the Schr\"odinger representation in the real setting to the FBI--Bargmann representation in the complex domain, conjugating the operators by a suitable globally defined metaplectic FBI transformation, see~\cite{Sj96},~\cite{HiSj15}. Once transported to the FBI transform side, pseudo\-differential operators in the Weyl quantization act on exponentially weighted spaces of entire holomorphic functions of the form
\begeq
\label{intro1}
H_{\Phi_0}(\comp^n) = {\rm Hol}(\comp^n) \cap L^2(\comp^n, e^{-2\Phi_0/h} L(dx)),
\endeq
where the weight function $\Phi_0$ is quadratic strictly plurisubharmonic, and $L(dx)$ is the Lebesgue measure on $\comp^n$. The purpose of the present work is to apply some of the $H_{\Phi}$--techniques to the study of semiclassical Weyl pseudodifferential operators with Gevrey symbols in the complex domain, showing a number of fairly general results concerning their symbolic and mapping properties. While the present work does not contain any applications to the study of propagation of singularities in Gevrey spaces, we expect the results established here to be useful in this respect and plan to return to these aspects in the near future. Let us now proceed to describe the precise assumptions and state the main results.

\bigskip
\noindent
Let $s>1$. The (global) Gevrey class ${\cal G}_b^s(\real^m)$ consists of all functions $u\in C^{\infty}(\real^m)$ such that there exist
$A>0$, $C>0$ such that for all $\alpha \in \nat^m$, we have
\begeq
\label{intr2}
\abs{\partial^{\alpha} u(x)} \leq A C^{\abs{\alpha}} (\alpha!)^s,\quad x\in \real^m.
\endeq
Let us also set ${\cal G}^s_0(\real^m) = {\cal G}_b^s(\real^m)\cap C^{\infty}_0(\real^m)$.

\medskip
\noindent
Associated to the quadratic form $\Phi_0$ in (\ref{intro1}) is the real linear subspace
\begeq
\label{intr3}
\Lambda_{\Phi_0} = \left\{\left(x,\frac{2}{i}\frac{\partial \Phi_0}{\partial x}(x)\right), \, x\in \comp^n\right\} \subset \comp^{2n} = \comp^n_x \times \comp^n_{\xi},
\endeq
which can be viewed as the image of the real phase space $\real^{2n}$ under a complex linear canonical transformation, see~\cite{Sj82},~\cite{Sj96},~\cite{HiSj15}. Identifying $\Lambda_{\Phi_0}$ linearly with $\comp^n_x$, via the projection map $\pi_x: \Lambda_{\Phi_0} \ni (x,\xi) \mapsto x\in \comp^n_x$, we may define the Gevrey spaces ${\cal G}^s_b(\Lambda_{\Phi_0})$, ${\cal G}^s_0(\Lambda_{\Phi_0})$. Given $a\in {\cal G}^s_b(\Lambda_{\Phi_0})$, for some $s>1$, and $u\in {\rm Hol}(\comp^n)$ such that $u(x) = {\cal O}_{h,N}(1) \langle{x\rangle}^{-N} e^{\Phi_0(x)/h}$ for all $N\in \nat$, let us introduce the semiclassical Weyl quantization of $a$ acting on $u$,
\begeq
\label{intr4}
a^w_{\Gamma}(x,hD_x) u(x) = \frac{1}{(2\pi h)^n}\int\!\!\!\!\int_{\Gamma(x)} e^{\frac{i}{h}(x-y)\cdot \theta} a\left(\frac{x+y}{2},\theta\right)u(y)\, dy\wedge d\theta.
\endeq
Here $0 < h \leq 1$ is the semiclassical parameter and $\Gamma(x)\subset \comp^{2n}_{y,\theta}$ is the natural integration contour given by
\begeq
\label{intr5}
\theta = \frac{2}{i} \frac{\partial \Phi_0}{\partial x}\left(\frac{x+y}{2}\right).
\endeq

\medskip
\noindent
In this work, we shall consider deformations of the standard weight $\Phi_0$, and to this end let $\Phi_1 \in C^{1,1}(\comp^n;\real)$ be such that
\begeq
\label{intr6}
\norm{\nabla^k(\Phi_1 - \Phi_0)}_{L^{\infty}({\bf C}^n)} \leq \frac{1}{C} h^{1 - \frac{1}{s}},\quad k = 0,1,2,
\endeq
for some $C>0$ sufficiently large. Our first main result is as follows --- see also Theorem \ref{main} and Theorem \ref{perturbationresult} below for a slightly more general statement.
\begin{theo}
\label{theo_main1}
Let $\displaystyle \omega = h^{1-\frac{1}{s}}$ and introduce the following $2n$--dimensional Lipschitz contours for $j=0,1$ and $x\in \comp^n$,
\begeq
\label{intr7}
\Gamma^{\Phi_j}_{\omega}(x): \quad \theta=\frac{2}{i}\frac{\partial \Phi_j}{\partial x}\left(\frac{x+y}{2}\right)
 + if_\omega(x-y),\quad y\in \comp^n,
\end{equation}
where
\begin{equation}
\label{intr8}
f_\omega({z})= \begin{cases}
                           \hskip10pt\overline{{z}}, \quad \, |{z}|\leq \omega,\\
                           {}\\
                           \displaystyle \omega\frac{\overline{{z}}}{|{z}|}, \quad |{z}| > \omega.
                          \end{cases}
\end{equation}
Let $a\in {\cal G}^s_b(\Lambda_{\Phi_0})$, for some $1 < s \leq 2$, and let $\widetilde{a}\in {\cal G}^s_b(\comp^{2n})$ be an almost holomorphic extension of $a$ such that ${\rm supp}\, \widetilde{a} \subset \Lambda_{\Phi_0} + B_{{\bf C}^{2n}}(0,C_0)$, for some $C_0>0$. We have for $j=0,1$,
\begin{multline}
\label{intr9}
a^w_{\Gamma}(x,hD_x) - \widetilde{a}^w_{\Gamma^{\Phi_j}_{\omega}}(x,hD_x) = {\cal O}(1)\, \exp\left(-\frac{1}{{\cal O}(1)}h^{-\frac{1}{s}}\right):\\
L^2(\comp^n, e^{-2\Phi_j/h} L(dx)) \rightarrow L^2(\comp^n, e^{-2\Phi_j/h} L(dx)),
\end{multline}
where the realization
$$
\widetilde{a}^w_{\Gamma^{\Phi_j}_{\omega}}(x,hD_x)u(x) = \frac{1}{(2\pi h)^n}\int\!\!\!\!\int_{\Gamma^{\Phi_j}_{\omega}(x)} e^{\frac{i}{h}(x-y)\cdot \theta} \widetilde{a}\left(\frac{x+y}{2},\theta\right)u(y)\, dy\wedge d\theta
$$
satisfies
\begeq
\label{intr10}
\widetilde{a}^w_{\Gamma^{\Phi_j}_{\omega}}(x,hD_x) = {\cal O}(1): H_{\Phi_j}(\comp^n) \rightarrow L^2(\comp^n, e^{-2\Phi_j/h} L(dx)).
\endeq
Here we have set, similarly to {\rm (\ref{intro1})},
$$
H_{\Phi_1}(\comp^n) = {\rm Hol}(\comp^n) \cap L^2(\comp^n, e^{-2\Phi_1/h} L(dx)).
$$
\end{theo}

\medskip
\noindent
In the range of Gevrey indices $s>2$, it turns out that we have to accept remainders that are larger than the ones in (\ref{intr9}), when obtaining uniformly bounded realizations of the operator $a^w_{\Gamma}(x,hD_x)$ on the weighted spaces $H_{\Phi_0}(\comp^n)$, $H_{\Phi_1}(\comp^n)$. The following is the second main result of this work.
\begin{theo}
\label{theo_main2}
Let $a\in {\cal G}^s_b(\Lambda_{\Phi_0})$, for some $s>2$, and let $\widetilde{a}\in {\cal G}^s_b(\comp^{2n})$ be an almost holomorphic extension of $a$ as in Theorem {\rm \ref{theo_main1}}. Let $\Gamma^{\Phi_j}_{h^{1/2}}(x)$ be the $2n$--dimensional Lipschitz contour defined as in {\rm (\ref{intr7})}, {\rm (\ref{intr8})}, with $\omega$ replaced by $h^{1/2} \geq \omega$. We have for $j=0,1$,
\begin{multline}
\label{intr11}
a^w_{\Gamma}(x,hD_x) - \widetilde{a}^w_{\Gamma^{\Phi_j}_{h^{1/2}}}(x,hD_x) = {\cal O}(1)\, \exp\left(-\frac{1}{{\cal O}(1)}h^{-\frac{1}{2s-2}}\right):\\
L^2(\comp^n, e^{-2\Phi_j/h} L(dx)) \rightarrow L^2(\comp^n, e^{-2\Phi_j/h} L(dx)),
\end{multline}
where
\begeq
\label{intr12}
\widetilde{a}^w_{\Gamma^{\Phi_j}_{h^{1/2}}}(x,hD_x) = {\cal O}(1): H_{\Phi_j}(\comp^n) \rightarrow L^2(\comp^n, e^{-2\Phi_j/h} L(dx)).
\endeq
\end{theo}

\bigskip
\noindent
As is seen in the statements of Theorem \ref{theo_main1} and Theorem \ref{theo_main2}, a crucial role in this work is played
by the existence of a Gevrey almost holomorphic extension of a symbol $a\in {\cal G}^s_b(\Lambda_{\Phi_0})$, off the maximally totally real linear subspace $\Lambda_{\Phi_0}\subset \comp^{2n}$ to a complex neighborhood. As discussed in Section \ref{ahol}, the existence of such an extension may be obtained by solving a Borel problem in the Gevrey space, relying on the work~\cite{Ca61} by Carleson. Alternatively, the existence of an extension $\widetilde{a}\in C^{\infty}_b(\comp^{2n})$ of $a\in {\cal G}^s_b(\Lambda_{\Phi_0})$ such that
\begin{equation}
\label{intr.13}
\abs{\overline{\partial} \widetilde{a}(\rho)} \leq
\mathcal{O}(1)\, \exp\left(-\frac{1}{\mathcal{O}(1)}\textrm{dist}\big(\rho,\Lambda_{\Phi_0}\big)^{-\frac{1}{s-1}}\right),\quad \rho \in \comp^{2n},
\end{equation}
may be obtained by adapting a construction of Mather~\cite{Ma71}, see also~\cite{DiSj99}, working with the Fourier inversion formula with a cutoff.

\bigskip
\noindent
We would like to emphasize that, as explained in Section \ref{psg}, replacing the Lipschitz contour $\Gamma^{\Phi_0}_{\omega}(x)$ in Theorem \ref{theo_main1} by a contour of the form
$$
\theta=\frac{2}{i}\frac{\partial \Phi_0}{\partial x}\left(\frac{x+y}{2}\right) + \frac{i}{C}\overline{(x-y)}, \quad C>0,
$$
natural in the holomorphic category~\cite{Sj82},~\cite{Sj96},~\cite{HiSj15}, leads only to remainder estimates of the form
$$
R = {\cal O}(1)\, \exp\left(-\frac{1}{{\cal O}(1)}h^{-\frac{1}{2s-1}}\right): H_{\Phi_0}(\comp^n) \rightarrow L^2(\comp^n, e^{-2\Phi_0/h} L(dx)),
$$
and thus, working with a "mixed" contour such as $\Gamma^{\Phi_0}_{\omega}(x)$, staying closer to $\Lambda_{\Phi_0}$, seems essential when obtaining accurate remainder estimates. The price that we have to pay for working with the contours $\Gamma^{\Phi_j}_{\omega}(x)$ in (\ref{intr7}), (\ref{intr8}) is that the realizations of our Gevrey pseudodifferential operators along such contours are uniformly bounded in the range $1 < s \leq 2$ only. Closely related to this is the well known observation~\cite{La88},~\cite{LaLa97} that while the class of operators of the form $a^w_{\Gamma}(x,hD_x)$, with $a\in {\cal G}^s_b(\Lambda_{\Phi_0})$, is stable under the composition, the standard asymptotic Weyl calculus does not lead to some sharp control of the remainders in the semiclassical expansions --- see also (\ref{comp.25}) below.

\medskip
\noindent
Let us conclude the introduction by mentioning several works where the Gevrey regularity questions were studied, which have provided some of the motivation for the present paper. The recent work~\cite{GeJe20} gives a detailed treatment of the Gevrey regularity framework on arbitrary real analytic compact manifolds, motivated by the microlocal study of dynamical zeta functions and trace formulas for Anosov flows.
In the context of scattering theory, Gevrey regularity questions were considered in~\cite{Rouleux},~\cite{GaZw}. We would finally like to refer to~\cite{Sj80_1, Sj80_2, Sj81} and to~\cite{Le84},~\cite{LaLa91} for results on the propagation of analytic and Gevrey singularities for boundary value problems, see also~\cite{Sj90}. To the best of our knowledge, the question whether the result of~\cite{Le84} on the non-diffraction
of Gevrey 3 singularities holds true in the complement of a Gevrey 3 obstacle is still open, see~\cite{LaLa97}.

\medskip
\noindent
The plan of the paper is as follows. Section 2 is devoted to the discussion of almost holomorphic extensions of Gevrey symbols. We also establish an approximate uniqueness of almost holomorphic extensions satisfying (\ref{intr.13}). In Section 3 we study semiclassical Gevrey pseudodifferential operators acting on $H_{\Phi}$--spaces, establishing Theorem \ref{theo_main1} and Theorem \ref{theo_main2}, using almost holomorphic extensions and contour deformations. The section is concluded by the discussion of the composition of semiclassical Gevrey operators,  by the methods of phase symmetries and contour deformations.

\bigskip
\noindent
{\bf Acknowledgements}. M.H. is very grateful to Andr\'as Vasy for a helpful discussion.

\section{Gevrey spaces and almost holomorphic extensions}\label{ahol}
\setcounter{equation}{0}

In this section we shall recall some well known facts concerning almost holomorphic extensions. One can consult \cite{FuNeRaSc19} for a recent very general treatment with plenty of references, in particular to the pioneering work of Dyn'kin~\cite{Dy70}. See also~\cite{MeSj75},~\cite{Dy93},~\cite{LaLa97}, and \cite{DiSj99}.

\medskip
\noindent
Let $\Omega \subset \comp^d$ be open and put $\Omega_R=\Omega \cap \real^d$. A function ${\widetilde u}\in C^{\infty}(\Omega )$ is called an almost holomorphic extension of $u\in C^{\infty}(\Omega_R )$ if

\begin{itemize}
\item[(i)] ${\widetilde u}_{\big|{y}=0}=u$, and

\item[(ii)] $\displaystyle \overline{\partial}{\widetilde u}=
\sum_{j=1}^{d}\,{{\partial}}_{\overline{z}_j}{\widetilde u}\ d\overline{z}_j$ is flat on ${y}=0$.
\end{itemize}
Here, we identify $\comp^d$ with $\real^{2d}$ in the usual way: $\comp^d\ni z=x+iy$, $(x,y)\in \real^{2d}$. Recall also that for $j=1,\ldots,d$, we have $\partial_{z_j}=\frac{1}{2}(\partial_{x_j}-i\partial_{y_j})$, ${\partial}_{\overline{z}_j}=\frac{1}{2}(\partial_{x_j}+i\partial_{y_j})$, $\partial_{z_j}-{\partial}_{\overline{z}_j}=-i\partial_{y_j}$ and
$\partial_{z_j}+{\partial}_{\overline{z}_j}=\partial_{x_j}$. If ${\widetilde u}$ is an almost holomorphic extension of $u$, the conditions
(i) and (ii) above determine the asymptotic (Taylor) expansion of ${\widetilde u}$ on ${y}=0$, i.e.,
\begin{equation*}
{\widetilde u}({x}+i{y})=\sum_{|\alpha|<N} \frac{i^{|\alpha|}}{\alpha !}
 u^{(\alpha)}({x}) {y}^\alpha +\mathcal{O}(|{y}|^N),
\end{equation*}
locally uniformly on $\mathrm{neigh\,}(\Omega_R,\Omega )$ for every $N\geq 1$. Here we write ``$\mathrm{neigh\,}(A,B)$'' as an abbreviation for
``some neighborhood of $A$ in $B$'' and, $u^{(\alpha)}=\partial^\alpha u$.

\medskip
\noindent
The function ${\widetilde u}$ is an almost holomorphic extension of $u$ if and only if ${\widetilde u}$ solves the Borel problem:
\begin{equation}\label{ahol.1}
\left(\partial_{{y}}^{\alpha}{\widetilde u}\right)_{\big|{y}=0}=i^{|\alpha|} u^{(\alpha)}, \qquad {\textrm{for all}}\quad \alpha\in {\bf{N}}^d.
\end{equation}
Indeed, we have already checked the necessity of \eqref{ahol.1}, and if \eqref{ahol.1}
is satisfied by ${\widetilde u}\in C^{\infty}(\comp^d)$ then
${\widetilde u}_{\big|{y}=0}=u$ and more generally
$\left(\partial_{{x}}^{\gamma}\partial_{{y}}^{\beta}{\widetilde u}\right)_{\big|{y}=0}
=i^{|\beta|} u^{(\beta+\gamma)}$ for all $\beta,\gamma\in \nat^d$. It follows that ${\partial}_{\overline{z}_j}{\widetilde u}$ is flat on ${y}=0$ as
$$
\left(\partial_{{x}}^{\gamma}\partial_{{y}}^{\beta}{{\partial}_{\overline{z}_j}}{\widetilde u}\right)_{\big|{y}=0}
=\frac{i^{|\beta|}}{2}\big(u^{(\beta+\gamma + e_j)}
+i^2 u^{(\beta+\gamma + e_j)}\big)=0.
$$
Here ${e_j}$ denotes the multi-index $(\delta_k^j)_{1\leq k\leq d}\in\nat^d$, where $\delta_k^j$ is the Kronecker delta.

\bigskip
\noindent
Let $\mathcal{U}$ be an open subset of $\real^d$, and let $s\geq 1$. The Gevrey space $\mathcal{G}^s(\mathcal{U})$ is the space of functions $u\in C^\infty(\mathcal{U})$ such that for every $K\Subset \mathcal{U}$, there exist $A>0$, $C>0$ such that
\begin{equation}
\label{ahol.G.1}
\big\vert{\partial^\alpha u(x)}\vert \leq A\, C^{|\alpha|} (\alpha!)^s,
\hbox{ for all }x\in K,\ \alpha\in{\bf{N}}^d.
\end{equation}
The class $\mathcal{G}^1(\mathcal{U})$ is the space of real analytic functions in $\mathcal{U}$, while for $s>1$, we have
$\mathcal{G}^s_0(\mathcal{U}):=\mathcal{G}^s(\mathcal{U}) \cap C^\infty_0(\mathcal{U}) \ne \{0 \}$, see~\cite[Theorem 1.3.5]{Ho85}. We also let ${\cal G}^s_b(\real^d) \subset {\cal G}^s(\real^d)$ be the space of functions $u\in C^{\infty}(\real^d)$ satisfying the Gevrey condition (\ref{ahol.G.1}) uniformly on all of $\real^d$: we have $u\in {\cal G}^s_b(\real^d)$ precisely when there exist $A>0$, $C>0$ such that for all $\alpha \in \nat^d$, we have
\begin{equation}
\label{ahol.G.1.1}
\abs{\partial^{\alpha} u(x)} \leq A C^{\abs{\alpha}} (\alpha!)^s,\quad x\in \real^d.
\end{equation}

\subsection{Almost holomorphic extensions via a result of Carleson}\label{Carl}
We assume now that $u\in {\mathcal{G}}^s_0(\real^d)$ with $s>1$. In view of the above remark, one has to solve
the Borel problem \eqref{ahol.1} in ${\mathcal{G}}^s(\real^{2d})$ in order to obtain an almost holomorphic extension ${\widetilde{u}}$ of $u$
in the same Gevrey class.

\medskip
\noindent
This may be achieved through the Carleson theorem with a suitable choice of the weight function, see \cite[Theorem 2 and Example 2]{Ca61}. This theorem is a corollary of a more general result which has allowed to resolve the issue known as  a ``universal moment problem''. See
\cite[Theorem 1]{Ca61}. We get
  \begin{prop}\label{Carl1}
Let $u\in {\cal G}^s_0(\real^d)$. Then $u$ has an almost
holomorphic extension $\widetilde{u}\in {\cal G}^s(\comp^d)$.
\end{prop}
Clearly, $\widetilde{u}$ vanishes to infinite order on $\real^d\setminus \mathrm{supp\,}(u)$. Let $\chi \in {\cal G}_0^s(\comp^d)$ be equal to $1$ near $\mathrm{supp\,}(u)$. Then $\chi \widetilde{u}\in {\cal G}_0^s(\comp^d)$ is also an almost holomorphic extension of
$u$. This gives the following variant:
\begin{prop}\label{Carl2}
Let $\Omega _R$, $\Omega $ be as above and let $u\in {\cal G}_0^s(\Omega _R)$. Then $u$ has an almost holomorphic extension
$\widetilde{u}\in {\cal G}_0^s(\Omega )$.
\end{prop}
With the help of Gevrey cutoffs and partitions of unity, we get the following variant:
\begin{prop}\label{Carl3}
Let $\Omega _R$, $\Omega $ be as above and let $u\in {\cal G}^s(\Omega _R)$. Then $u$ has an almost holomorphic extension
$\widetilde{u}\in {\cal G}^s(\Omega )$.
\end{prop}

\medskip
\noindent
In the case when $u\in {\cal G}^s_b(\real^d)$, we obtain that there exists an almost holomorphic extension
$\widetilde{u}\in {\cal G}^s_b(\comp^d)$, which is supported in a bounded tubular neighborhood of $\real^d \subset \comp^d$.

\bigskip
\noindent
Let $\widetilde{u}\in {\cal G}_0^s(\comp^d)$ be an almost holomorphic extension of $u \in {\cal G}_0^s(\real^d)$. In view of Taylor's formula, there exist $C>0,\ A>0$ such that
\begin{equation}
\label{ahol.2}
\Big\vert\big(\overline{\partial}\widetilde{u}\big)({z})\Big\vert \leq
C A^N \big(N!)^{s-1} \big|\mathrm{Im}({z})\big|^N, \qquad
{\textrm{for all}}\quad N\in \nat.
\end{equation}
Here, $\mathrm{Im}({z})$ denotes the imaginary part of $z\in \comp^d$. Taking the infimum over $N$ one gets
\begin{equation}
\label{ahol.3}
\Big\vert\big(\overline{\partial}\widetilde{u}\big)({z})\Big\vert={\mathcal{O}(1)}\,
\exp\left(-\frac{1}{\mathcal{O}(1)}\big|\mathrm{Im}({z})\big|^{-\frac{1}{s-1}}\right),
\end{equation}
where ${\cal O}(1)$ denotes a number whose modulus is bounded by some large positive constant, positive when appearing in a denominator.

\medskip
\noindent
Now if $\widetilde{u}_1, \widetilde{u}_2$ are two almost holomorphic extensions of $u$ as above, we have in view of
\eqref{ahol.1} and of Taylor's formula $\big\vert \big(\widetilde{u}_1-\widetilde{u}_2\big)({z})\big\vert
\leq C A^N \big(N\,!)^{s-1} \big|\mathrm{Im}({z})\big|^N$, for all $N\in \nat$. Taking the infimum over $N$, we get,
\begin{equation}
\label{ahol.4}
\Big\vert\big(\widetilde{u}_1-\widetilde{u}_2\big)({z})\Big\vert = {\mathcal{O}(1)}
\exp\left(-\frac{1}{\mathcal{O}(1)}\big|\mathrm{Im}({z})\big|^{-\frac{1}{s-1}}\right).
\end{equation}
In the analogous context of Proposition \ref{Carl3}, we get (\ref{ahol.4}) locally uniformly on $\Omega $.

\subsection{Fourier transforms}\label{Ftf}
The Fourier transform of a function $u\in {\cal S}(\real^d)$ is given by
$$
\mathcal{F}u(\xi )=
\widehat{u}(\xi)=\int_{{\bf R}^d}e^{-i y\cdot \xi }u(y)dy, \hbox{ where
} y\cdot \xi =\sum_{j=1}^d y_j\xi _j,
$$
and we have the Fourier inversion formula,
$$
u(x)=(2\pi)^{-d}\int_{{\bf{R}}^d} e^{i x\cdot \xi}\widehat{u}(\xi)\,d\xi.
$$

\medskip
\noindent
If $s>1$ and $u\in\mathcal{G}^s_0(\real^d)$, there exists a constant $C>0$
such that
\begin{equation}\label{ahol.G.2}
\big\vert{\widehat{u}(\xi)}\big\vert\leq C\,\exp\left(-\frac{1}{C}\vert{\xi}\vert^{\frac{1}{s}}\right),
{\hbox{ for every }} \xi \in \real^d.
\end{equation}
Indeed, if $N\in 2\nat$, we have
${\cal F}\left( (1-\Delta)^{\frac{N}{2}}u \right)
(\xi)= (1+|\xi|^2)^{\frac{N}{2}}\widehat{u}(\xi)$, and it follows that
$\big\vert{\widehat{u}(\xi)}\big\vert\leq A^{N+1}(1+\vert{\xi}\vert^2)^{-\frac{N}{2}}(N!)^s$.
It then suffices to choose $N\sim \left(\frac{\vert{\xi}\vert}{A}\right)^{\frac{1}{s}}$
and apply Stirling's formula.

\medskip
\noindent
Conversely, from the Fourier inversion formula, we see that if $u\in {\cal S}(\real^d)$ and \eqref{ahol.G.2} holds, then $u\in\mathcal{G}^s_{{b}}(\real^d)$.

\subsection{Almost holomorphic extensions in the spirit of Mather}
\label{Math}
Let $u\in {\mathcal{G}}^s_0(\real^d)$ and let $\widetilde{u}$ be an almost holomorphic extension of $u$ satisfying \eqref{ahol.3}.
If $\xi \in \real^d$, $|\xi| \ge 1$, we may assume after an orthogonal change of coordinates, that $\xi/|\xi|=e_d=(0,\ldots,0,1)\in \real^d$ and in view of Stokes' formula we have
$$
\widehat{u}(\xi)=\int_{\mathbf{\Pi}_{d}^{-}}e^{-i{z}\cdot\mathbf{\xi}}\,
\big({\partial_{\overline{z}_d}}\widetilde{u}\big)(x',{z})
\,dx'\wedge d\overline{z}_d\wedge dz_d,
$$
where $\mathbf{\Pi}_{d}^{-}=\big\{(x',{z_d})\in\real^{d-1}\times
\comp;\, \textrm{Im}(z_d)\leq 0\big\}$.

\medskip
\noindent
From \eqref{ahol.3} we get
$\big\vert{\partial}_{\overline{z}_d}\widetilde{u}({z})\big\vert
={\mathcal{O}(1)}
 \exp\Big(-\frac{1}{\mathcal{O}(1)}|\mathrm{Im}({z})|^{-\frac{1}{s-1}}\Big)$ for all
${z}\in \mathbf{\Pi}_{d}^{-}$.
Then,
$$
\big\vert\widehat{u}(\xi)\big\vert=\mathcal{O}(1)
\exp\left(-\inf_{0\leq t < \infty} \big(t|\xi|+\widetilde{C}t^{-\frac{1}{s-1}}\big)\right),
$$
for some $\widetilde{C}>0$. A straightforward calculation shows that the infimum over the positive half axis is attained at
$$
t_{\xi}=\left(\frac{\widetilde{C}}{s-1}\right)^{\frac{s-1}{s}}
\vert\xi\vert^{-\frac{s-1}{s}},
$$
and the corresponding value of the infimum is equal to
$$
\widetilde{C}^{\frac{s-1}{s}}\frac{s}{(s-1)^{\frac{s-1}{s}}} |\xi|^{\frac{1}{s}},
$$
which implies that
\begin{equation}\label{ahol.5}
\big\vert\widehat u(\xi)\big\vert=
\mathcal{O}(1)\, \exp\left(-\frac{1}{\mathcal{O}(1)}|\xi|^{\frac{1}{s}}\right),
\quad \hbox{ for all }\, \xi\in \real^d.
\end{equation}

\medskip
\noindent
Conversely, assume that \eqref{ahol.5} holds (which is the case when $u\in {\mathcal{G}}^s_0(\real^d)$). Following~\cite{Ma71}, we look for an extension $\widetilde{u}(z)=\widetilde{u}(x+i y)$ by truncation in Fourier's inversion formula. Let us start with the formal identity,
$$
\widetilde{u}(z)=\frac{1}{(2\pi)^d}\int_{{\bf R}^d} e^{i\xi\cdot x - \xi\cdot y} \widehat{u}(\xi)\,d\xi=
\mathcal{O}(1)\int_{{\bf R}^d} e^{|\xi| |y|-\frac{|\xi|^{\frac{1}{s}}}{C}}\,d\xi,
$$
where $C$ is a positive constant. For $|y|\le |\xi |^{-(s-1)/s}/(2C)$, $|\xi |\ge 1$, the integrand in the last integral
is $\le \exp \left( -|\xi |^{1/s}/(2C) \right)$. Let $\psi \in C_0^\infty
([0,1/(2C)))$ be equal to $1$ near $0$ and set
\begin{equation}\label{ahol.6}
\widetilde{u}(z)=\frac{1}{(2\pi )^d}\int_{{\bf R}^d}\psi \left(|y|
\, |\xi |^{\frac{s-1}{s}} \right)\widehat{u}(\xi )e^{(ix-y)\xi }d\xi
\end{equation}
so that $\widetilde{u}\in C^\infty (\comp^d)$. We have
\begin{equation}\label{ahol.7}
  \overline{\partial }\widetilde{u}(z)=\frac{1}{(2\pi )^d }
  \int (\partial _r\psi )\left(|y|\, |\xi |^{\frac{s-1}{s}} \right)
  |\xi |^{\frac{s-1}{s}}{\partial}_{\overline z}(|y|) \widehat{u}(\xi )
  e^{(ix-y)\xi }d\xi .
\end{equation}
Here, ${\partial}_{\overline z}(|y|)={\cal O}(1)$ and on the support of
$(\partial _r\psi )\left(|y|\, |\xi |^{(s-1)/s} \right)$ we have for
some constant $\widetilde{C}>2C$,
$$
\frac{1}{\widetilde{C}}\le |y|\, |\xi |^{\frac{s-1}{s}}\le \frac{1}{2C},
$$
i.e.\
\begin{equation}\label{ahol.8}
\left(\frac{1}{\widetilde{C}|y|} \right)^{\frac{1}{s-1}}\le |\xi |^{\frac{1}{s}}\le
\left(\frac{1}{2C|y|} \right)^{\frac{1}{s-1}} .
\end{equation}
We conclude that for $\widetilde{u}$ given in (\ref{ahol.6}),
\begin{equation}
\label{ahol.9}
|\overline{\partial }\widetilde{u}(z)|\le {\cal O}(1)\,\exp \left( -|\Im
  z|^{-\frac{1}{s-1}}/{\cal O}(1) \right),\ \ |\Im z|\le {\cal O}(1).
\end{equation}
We get the same estimates for $\partial _z^\alpha \partial_{\overline{z}}^\beta \overline{\partial }\widetilde{u}$, for all
$\alpha ,\beta \in \nat^d$.

\subsection{Approximate uniqueness via a Carleman estimate}
\label{uni}
The existence of almost holomorphic extensions was established in the previous subsections by appealing to the result of Carleson~\cite{Ca61} and also by Mather's method~\cite{Ma71}. We shall consider here almost holomorphic extensions which are not necessarily Gevrey and we shall get approximate uniqueness through a Carleman estimate of H\"ormander type.

\medskip
\noindent
To get the desired uniqueness estimate at a given point $z\in \comp^d$
with ${\mathrm{Im}\,z} \ne 0$, we can restrict the attention to the complex line
\begin{equation}
\label{uni.0}
\mathrm{Re}\,z + \comp\,\mathrm{Im}\, z,
\end{equation}
so it will suffice to consider the one-dimensional case.

\medskip
\noindent
Let $\Omega \Subset \comp$ be open connected with smooth boundary. Let $\phi \in C^\infty (\overline{\Omega };\real)$ be
strictly subharmonic,
\begin{equation}
\label{uni.1}
\partial^2_{z\,\overline{z}}\phi>0\quad\hbox{ on }\,\,\overline{\Omega}.
\end{equation}
To the operators
\begeq
\label{uni.1.1}
P:=D_{\overline{z}}=\frac{1}{i}\partial _{\overline{z}}=\frac{1}{2}(D_x+iD_y)\quad \hbox{ and }\quad P^*:=D_z=\frac{1}{i}\partial_z=\frac{1}{2}(D_x-iD_y),
\endeq
we associate the symbols
$$
\overline{\zeta } =\frac{1}{2}(\xi +i\eta )\quad \hbox{ and }\quad \zeta  =\frac{1}{2}(\xi -i\eta ),
$$
respectively. Here $z=x+iy$ with $(x,y)\in\real^2$. We introduce the corresponding conjugated operators:
$$
P_\phi =e^{\phi}Pe^{-\phi}\quad \hbox{ and }\quad P_\phi^* =e^{-\phi}P^*e^{\phi}.
$$
More explicitly,
\begin{equation}
\label{uni.2}
\begin{cases}
P_\phi =\frac{1}{i}(\partial_{\overline{z}}-\partial _{\overline{z}}\phi ),\\
P^*_\phi =\frac{1}{i}(\partial_{z}+\partial_{z}\phi ).
\end{cases}
\end{equation}
We think of $P_\phi $ as $P$, acting on $e^{-\phi}L^2(\Omega)=L^2\big(\Omega ;e^{2\phi(z)}\,L(dz)\big)$, where $L(dz)$ is the Lebesgue measure on ${\bf{C}}$. Formally, we have
\begin{equation}\label{uni.3}
[P_\phi ^*,P_\phi ]=-[\partial _z+\partial _z\phi , \partial
_{\overline{z}}-\partial _{\overline{z}}\phi ]=2\partial^2_{z\,
\overline{z}}\phi =\frac{1}{2} \Delta _{x,y} (\phi ).
\end{equation}
For $v\in H^1_0(\Omega)$ (i.e. of class $H^1(\Omega)$, vanishing on the boundary), we have, using the $L^2(\Omega)$--norm and the scalar product,
\begin{equation}
\label{uni.4}
\|P_\phi v\|^2=(P_\phi ^*P_\phi v|v)=([P_\phi ^*,P_\phi ]v|v)
+\|P_\phi ^*v\|^2,
\end{equation}
and (\ref{uni.3}) leads to the Carleman estimate
\begin{equation}\label{uni.5}
2(\partial^2_{z\,\overline{z}}\phi\, v|v)\le \|P_\phi v\|^2,
\end{equation}
i.e.
\begin{equation}\label{uni.6}
\sqrt{2}\,\Big\|\big(\partial^2_{z\,\overline{z}}\phi\big)^{\frac12}\,v\Big\|\le \Big\|P_\phi\,v\Big\|,
\end{equation}
or after removing the conjugation,
\begin{equation}\label{uni.7}
\sqrt{2}\,\Big\|e^{\phi}\big(\partial^2_{z\,\overline{z}}\phi\big)^{\frac12}
\,v\Big\|\le \Big\| e^{\phi}\partial_{\overline{z}}\, v\Big\|,\quad \hbox{ for }\ v\in H_0^1(\Omega).
\end{equation}
We may notice here that in H\"ormander's approach to $\overline{\partial}$, we work in the weighted space $e^{\phi}L^2(\Omega)$ and get a priori estimates for $P_\phi^*$, leading to existence results for the operator $P_\phi $, see~\cite[Chapter 4]{Horm_conv}.

\bigskip
\noindent
We shall next discuss the choice of $\phi$. Assume that $\Omega$ is contained in the open upper half-plane with
$\overline{\Omega } \cap \real = [-1,1]$. We know that a $\mathcal{G}^s$ function $u$, defined near $[-1,1]$, has an extension
$\widetilde{u}\in \mathcal{C}^\infty (\overline{\Omega })$, which satisfies
\begin{equation}
\label{uni.8}
\overline{\partial }\widetilde{u}(z)=\mathcal{O}(1)
\exp \left(-\frac{1}{C_0}{\left(\mathrm{Im}\, z\right)^{-\frac{1}{s-1}}}\right),
\end{equation}
for some $C_0>0$. In the following we shall assume that $C_0=1$ for simplicity. In order to apply (\ref{uni.6}), (\ref{uni.7}), we would like to have a suitable modification of
\begin{equation}
\label{uni.9}
\phi (z)=\left(\mathrm{Im}\, z\right)^{-\frac{1}{s-1}}.
\end{equation}
Recalling that
$\partial^2_{z\,\overline{z}}=\frac{1}{4}\Delta_{\mathrm{Re}\,z,\mathrm{Im}\,z}$,
we compute $\displaystyle \partial _{\mathrm{Im}\,z}\phi = -(s-1)^{-1}\big({\mathrm{Im}\, z}\big)^{-\left(1 + \frac{1}{s-1}\right)}$ and then
\begin{equation}\label{uni.10}
\Delta_{\mathrm{Re}\,z,\mathrm{Im}\,z}
\phi=\partial^2_{\mathrm{Im}\,z,\mathrm{Im}\,z}\phi
=\frac{s}{(s-1)^2}\left({\mathrm{Im}\, z}\right)^{-\left(2 + \frac{1}{s-1}\right)}>0.
\end{equation}
We would like to apply (\ref{uni.7}) to the difference $\widetilde{u}_1-\widetilde{u}_2$ of
two almost holomorphic extensions of the same function $u$, both
satisfying (\ref{uni.8}) and run into two technical difficulties:

\begin{itemize}
\item[(i)] the function $\phi $ in (\ref{uni.9}) is not smooth up to the real part of the boundary of
  $\Omega $,
\item[(ii)] the difference $(\widetilde{u}_1-\widetilde{u}_2)$ does not vanish on all of $\partial \Omega $,
    but only on $\partial \Omega \cap \real$.
\end{itemize}

\medskip
\noindent
The first difficulty is easy to resolve by replacing $\phi $ by $\phi_\varepsilon (z)=\phi (z+i\varepsilon )$ and letting $\varepsilon$ tend to $0$. To resolve the second difficulty, one can multiply $(\widetilde{u}_1-\widetilde{u}_2)$ by a cutoff function that vanishes near
$\partial \Omega \setminus ]-1,1[$ and we then need to modify $\phi $ in this region.

\medskip
\noindent
In general, let $f(z)$ be smooth and real valued, defined near some point $z_0\in \comp$ where $f(z_0)=0$ and $df(z_0)\ne 0$. Consider
\begin{equation}
\label{uni.11}
\phi (z)=f(z)^{-\frac{1}{s-1}}
\end{equation}
in $\big\{z\in \mathrm{neigh\,}(z_0,\comp);\, f(z)>0\}$.
Then
$$
\partial_z\phi (z)=-(s-1)^{-1}f(z)^{-\left(1 + \frac{1}{s-1}\right)}\, \partial_z f(z),
$$
\begin{equation}
\label{uni.12}
\partial_{z\,\overline{z}}^2\phi(z)=
f(z)^{-\left(2 + \frac{1}{s-1}\right)} \frac{s}{(s-1)^2}
\left(|\partial _zf|^2-\frac{s-1}{s}f(z)\, \partial
  _{z\,\overline{z}} f(z) \right),
\end{equation}
generalizing (\ref{uni.10}), where $f(z)$ was equal to $\mathrm{Im}\, z$.

\medskip
\noindent
Let $\psi \in C_0^\infty \big(]-1,1[;[0,1]\big)$ be equal to one on $[-\frac{1}{2},\frac{1}{2}]$ and $<1$ outside that interval, let $g(t)=1-\psi (t)$, and put
\begin{equation}\label{uni.13}
f_\varepsilon (z) =\mathrm{Im}\,z + ag(\mathrm{Re}\,z)+\varepsilon,
\end{equation}
\begin{equation}\label{uni.14}\phi _\varepsilon (z)=f_\varepsilon
  (z)^{-\frac{1}{s-1}}.\end{equation}
Here $0<a\ll 1$ is fixed and $\varepsilon \ge 0$ is a small parameter. We
notice that

\begin{itemize}
\item[1)]  $\phi _\varepsilon \in C^\infty (\Omega )$ when $\varepsilon \ge0$,
$\phi _\varepsilon \in C^\infty (\overline{\Omega })$ when $\varepsilon >0$,

\item[2)] the function $\varepsilon \longmapsto f_\varepsilon $ is increasing, while
  $\varepsilon\longmapsto \phi_\varepsilon$ is decreasing.

\item[3)] We have by (\ref{uni.12}) that
\begin{equation}
\label{uni.15}
\partial^2_{z\,\overline{z}}\phi_\varepsilon (z)\asymp f_\varepsilon (z)^{-\left(2 + \frac{1}{s-1}\right)},
\end{equation}
uniformly for $(z,\varepsilon )\in \Omega \times ]0,\varepsilon _0]$, for some $\varepsilon _0>0$.
\end{itemize}

\medskip
\noindent
By 2), we have
$$
\phi _\varepsilon \le \phi _0 = \Big(\mathrm{Im}\, z+a g(\mathrm{Re}z)\Big)^{-\frac{1}{s-1}},
$$
and $\phi _0 \in C^\infty (\overline{\Omega }\setminus [-\frac12,\frac12])$. Let $\chi \in \mathcal{C}^\infty(\overline{\Omega })$  be
equal to one near $\mathrm{supp\,}\psi $ and vanish near $\partial \Omega \setminus ]-1,1[$.

\medskip
\noindent
Let
$\widetilde{u}_1,\widetilde{u}_2\in \mathcal{C}^\infty(\overline{\Omega })$ be two almost holomorphic extensions of the same
$\mathcal{G}^s$ function $u$, defined near $\overline{\Omega }\cap \real$, which satisfy (\ref{uni.8}) (with $C_0=1$ for simplicity), so
that
\begin{equation}\label{uni.16}
\overline{\partial }(\widetilde{u}_1-\widetilde{u}_2)=\mathcal{O}(1)\exp
\left(-\big({\mathrm{Im}\, z}\big)^{-\frac{1}{s-1}} \right) \hbox{ in }\overline{\Omega },
\end{equation}
\begin{equation}\label{uni.17}
\widetilde{u}_1-\widetilde{u}_2=0\,\, \hbox{ on }\,\, [-1,1].
\end{equation}
With $\chi$ as above, let
\begin{equation}\label{uni.18}
v:=\chi (\widetilde{u}_1-\widetilde{u}_2)\in \mathcal{C}^\infty (\overline{\Omega }).
\end{equation}
Then
\begin{equation}\label{uni.19}
{{v}_\vert}_{\partial \Omega }=0,
\end{equation}
\begin{equation}\label{uni.20}
\partial _{\overline{z}}v=(\widetilde{u}_1-\widetilde{u}_2)\partial _{\overline{z}}\chi +\chi
\partial_{\overline{z}}(\widetilde{u}_1-\widetilde{u}_2)=\mathcal{O}(1)e^{-\phi _0}\hbox{ in }\Omega .
\end{equation}
Combining this with (\ref{uni.7}), (\ref{uni.15}), (\ref{uni.16}) and
letting $\varepsilon \to 0$, we get
\begin{equation}\label{uni.21}
\Big\| \big(f_0(z)\big)^{-\frac{1}{2}\left(2+\frac{1}{s-1} \right)}e^{\phi _0}\,v\Big\|_{L^2(\Omega
)}\le \mathcal{O}(1) \Big\| e^{\phi _0}\partial
_{\overline{z}}\,v\Big\|_{L^2(\Omega )}.
\end{equation}

\medskip
\noindent
Let $W$ be a small complex neighborhood of $[-1/2,1/2]$ so
that
$$
\phi _0(z)=\big(\mathrm{Im}\, z\big)^{-\frac{1}{s-1}}\hbox{ and }\chi =1 \hbox{ in }W.
$$
Then
\begeq
\label{uni.22}
\left\|\left(\mathrm{Im}\, z\right)^{-\frac{1}{2}\left(2+\frac{1}{s-1}\right)}
\exp  \left({\left(\mathrm{Im}\, z\right)^{-\frac{1}{s-1}}}\right)
(\widetilde{u}_1-\widetilde{u}_2)
\right\|_{L^2(W)}
\leq  \mathcal{O}(1) \left\| e^{\phi _0}\partial
_{\overline{z}}\,v\right\|_{L^2(\Omega)}.
\endeq
By (\ref{uni.20}), the right hand side of \eqref{uni.22} is $\mathcal{O}(1)$. More explicitly,
from (\ref{uni.22}), (\ref{uni.20}), and the fact that $\phi _0$ is bounded on $\mathrm{supp\,}(\overline{\partial }\chi )$ and $\le \phi $, we have
\begin{multline*}
\left\|\left(\mathrm{Im}\, z\right)^{-\frac{1}{2}\left(2+\frac{1}{s-1} \right)}
\exp  \left({\left(\mathrm{Im}\, z\right)^{-\frac{1}{s-1}}}\right)
(\widetilde{u}_1-\widetilde{u}_2)
\right\|_{L^2(W)} \\
\leq
\mathcal{O}(1)   \left(\left\|
\exp \left(({\mathrm{Im}\, z})^{-\frac{1}{s-1}} \right) \overline{\partial }(\widetilde{u}_1-\widetilde{u}_2)
\right\|_{L^2(\Omega )} +\left\|
\widetilde{u}_1-\widetilde{u}_2
\right\|_{L^2(\mathrm{supp\,}\overline{\partial }\chi )}\right),
\end{multline*}
and we have proved the following slightly more general statement:
\begin{lemma}
\label{uni1}
Let $\Omega \Subset \comp$ be open with smooth boundary, contained in the open upper half plane, with
$\overline{\Omega }\cap\real = [-1,1]$. Let $\phi $ be given by {\rm (\ref{uni.9})}. Then there exists an open neighborhood $W$ of $[-1/2,1/2]$ in $\overline{\Omega }$ such that
\begin{equation}
\label{uni.23}
\left\| {(\mathrm{Im}\,z)}^{-\frac{1}{2}\left(2+\frac{1}{s-1}\right)} e^\phi
    (\widetilde{u}_1-\widetilde{u_2})\right\|_{L^2(W)}\leq {{\cal O}(1)} \left(
\left\| e^\phi \overline{\partial }(\widetilde{u}_1-\widetilde{u}_2)
\right\|_{L^2(\Omega )}+\left\| \widetilde{u}_1-\widetilde{u}_2
\right\|_{L^2(\Omega )}
  \right),
\end{equation}
for all $(\widetilde{u}_1,\widetilde{u}_2)\in (H^1(\Omega ))^2$ with
$\widetilde{u}_1=\widetilde{u}_2$ on $\overline{\Omega}\cap \real$ and $\overline{\partial }(\widetilde{u}_1-\widetilde{u}_2)\in
e^{-\phi }L^2(\Omega )$. \end{lemma}

\medskip
\noindent
This applies to the case when $\widetilde{u}_1$, $\widetilde{u}_2$ are two almost holomorphic extensions of  the same function $u\in {\cal G}^s(\mathrm{neigh\,}([-1,1],\real))$, satisfying (\ref{uni.8}) with $C_0=1$.

\bigskip
\noindent
We observed after (\ref{ahol.9}) that if $u\in {\cal G}_0^s(\real^d)$, and $\widetilde{u}$ is given in
(\ref{ahol.6}), then for all $\alpha ,\beta \in \nat^d$, there exists $C_{\alpha ,\beta }>0$ such that
\begin{equation}
\label{uni.24}
|\overline{\partial }\partial _z^\alpha \partial _{\overline{z}}^\beta
\widetilde{u}(z)|\le C_{\alpha ,\beta }\exp \left(-|\mathrm{Im}\, z|^{\frac{1}{s-1}}/C_0 \right),
\end{equation}
where $C_0>0$ is independent of $\alpha ,\beta $.

\medskip
\noindent
In the one-dimensional case, if $\widetilde{u}_1$, $\widetilde{u}_2$ are two almost holomorphic extensions of the same
function $u\in {\cal G}_0^s$, satisfying (\ref{uni.24}) (assuming for simplicity that $C_0=1$), we can apply (\ref{uni.23}) with $\widetilde{u}_j$ replaced by $\partial
_z^\alpha \partial _{\overline{z}}^\beta \widetilde{u}_j$ and see that
\begin{equation}
\label{uni.25}
\left\|(\mathrm{Im}\,z)^{-\frac{1}{2}(2+\frac{1}{s-1})}e^\phi \partial
  _z^\alpha \partial _{\overline{z}}^\beta
  (\widetilde{u}_1-\widetilde{u}_2) \right\|_{L^2(W)}\le {\cal
  O}_{\alpha ,\beta }(1).
\end{equation}

\medskip
\noindent
Remaining in the one-dimensional case, we  shall next show how to get from (\ref{uni.25}) an estimate of a weighted $L^\infty $-norm, having in mind that if $u\in H^2(D(0,1))$, then in view of the Sobolev embedding theorem, we have
\begin{equation}
\label{uni.26}
\| u\|_{L^\infty (D(0,1/2))} \le {\cal O}(1)\| u\|_{H^2(D(0,1))}.
\end{equation}
Because of the presence of exponential weights in (\ref{uni.25}), we shall work in very small discs $D(z,r)$ with the property that $e^{\phi (\zeta)}\asymp e^{\phi (z)}$ for $\zeta \in D(z,r)$. We have $\nabla \phi(z)
={\cal O}(1)(\mathrm{Im}\, z)^{-1-1/(s-1)}$ and $\phi (\zeta)-\phi
(z)={\cal O}(1)$ if $|\zeta-z|\leq (\mathrm{Im}\,z)^{1+1/(s-1)}$ and $0< \mathrm{Im}\, z \ll 1$. Thus, with
\begin{equation}\label{uni.27}
r=r(z) = (\mathrm{Im}\,z)^{1+\frac{1}{s-1}},
\end{equation}
we have
\begin{equation}\label{uni.28}
e^{\phi (\zeta)} \asymp e^{\phi (z)} \hbox{ when } 0 < \mathrm{Im}\, z \ll 1,\ \zeta \in D(z,r),
\end{equation}
and we have
$$
(\mathrm{Im}\,\zeta)^{-\frac{1}{2}\left(2+\frac{1}{s-1}\right)} \asymp
(\mathrm{Im}\, z)^{-\frac{1}{2}\left(2+\frac{1}{s-1}\right)},
$$
in the same set. For $\zeta \in D(z,r)$, write $\zeta = z+rw$, $w\in D(0,1)$. Using that
$$
\partial_{\zeta} = r^{-1}\partial _w,\ \partial_{\overline{\zeta}} = r^{-1}\partial _{\overline{w}},\ L(d\zeta)=r^2L(dw),
$$
and replacing $W$ in (\ref{uni.25}) by the smaller set $D(z,r)$ (so we have to take $z$ in a slightly shrunk copy of $W$) we get from (\ref{uni.25}),
\begin{equation}\label{uni.29}
(\mathrm{Im}\, z)^{-\frac{1}{2}\left(2+\frac{1}{s-1}\right)}e^{\phi
  (z)}r(z)^{1-\alpha -\beta }\left\| \partial _w^\alpha
  \partial_{\overline{w}}^\beta (\widetilde{u}_1-\widetilde{u}_2)
\right\|_{L^2(D(0,1))}
\le {\cal O}_{\alpha ,\beta }(1),
\end{equation}
where
$\widetilde{u}_1-\widetilde{u}_2=(\widetilde{u}_1-\widetilde{u}_2)(z+rw)$ is viewed as a function of $w$. Using (\ref{uni.29}) for
$\alpha +\beta \le 2$, we get a bound for $\|\widetilde{u}_1-\widetilde{u}_2\|_{H^2(D(0,1))}$ (in the $w$-variable) and with (\ref{uni.26}), we get
$$
(\mathrm{Im}\,z)^{-\frac{1}{2}\left(2+\frac{1}{s-1}\right)}e^{\phi(z)}r(z)\left\| \widetilde{u}_1-\widetilde{u}_2
\right\|_{L^\infty (D(0,1/2))}
\le {\cal O}(1).
$$
With $w=0$, we obtain
\begin{equation}
\label{uni.30}
(\mathrm{Im}\, z)^{-\frac{1}{2}\left(2+\frac{1}{s-1}\right)}e^{\phi
  (z)}r(z)\left| \widetilde{u}_1(z)-\widetilde{u}_2(z)
\right|
\le {\cal O}(1),
\end{equation}
uniformly on $W$ (after a slight decrease of $W$ or a slight increase of the original $W$). Using (\ref{uni.25}) for higher derivatives,
gives an extension,
\begin{equation}
\label{uni.31}
(\mathrm{Im}\, z)^{-\frac{1}{2}\left(2+\frac{1}{s-1}\right)}e^{\phi(z)}r(z)\left| \partial _z^\alpha \partial _{\overline{z}}^\beta (\widetilde{u}_1-\widetilde{u}_2)(z)
\right|
\le {\cal O}(1),
\end{equation}
uniformly on $W$, for every $(\alpha ,\beta )\in {\bf{N}}^2$. Here we notice that by (\ref{uni.27}),
\begin{equation}
\label{uni.32}
(\mathrm{Im}\, z)^{-\frac{1}{2}\left(2+\frac{1}{s-1} \right)}r(z)= (\mathrm{Im}\, z)^{\frac{1}{2(s-1)}}.
\end{equation}

\bigskip
\noindent
We now return to the $d$-dimensional case and apply the observation around (\ref{uni.0}) about the reduction to the one-dimensional case. From this and (\ref{uni.31}), (\ref{uni.32}), we get the main result of this subsection,
\begin{theo}
\label{uni2}
Let $u\in {\cal G}_0^s(\real^d)$, let $N_0\in {\bf{N}}$, and let $\widetilde{u}_1$, $\widetilde{u}_2$ be two almost holomorphic
extensions of $u$ such that (cf.\ {\rm (\ref{ahol.9})} and the subsequent remark) for all $\alpha ,\beta \in \nat^d$ with $|\alpha
|+|\beta |\le N_0+2$, we have for $j=1,2$,
\begin{equation}
\label{uni.33}
\left|\partial _z^\alpha \partial _{\overline{z}}^\beta
  \overline{\partial }\widetilde{u}_j(z)\right|
\le {\cal O}(1)\exp \left(-|\mathrm{Im}\, z|^{\frac{1}{s-1}}/C_0 \right),\ \
|\mathrm{Im}\,z|\le 1/{\cal O}(1),
\end{equation}
where $C_0>0$. Then,
\begin{equation}\label{uni.34}
\left|\partial _z^\alpha \partial _{\overline{z}}^\beta (\widetilde{u}_1(z)-\widetilde{u}_2(z))\right|
\le {\cal O}(1) |\mathrm{Im}\, z|^{-\frac{1}{2(s-1)}}\exp \left(-|\mathrm{Im}\, z|^{\frac{1}{s-1}}/C_0 \right),\ \
|\mathrm{Im}\,z |\le 1/{\cal O}(1),
\end{equation}
for all $\alpha ,\beta \in \nat^d$ with $|\alpha |+|\beta |\le N_0$. We get the same conclusion for $u\in {\cal G}^s_b(\real^d)$.
\end{theo}

\section{Pseudodifferential operators with Gevrey symbols in the complex domain}
\label{psg}
\setcounter{equation}{0}
\subsection{Almost holomorphic extensions and contour deformations}
\label{cds}
In the beginning of this subsection, we shall recall, following~\cite{Sj96},~\cite{HiSj15}, some basic facts concerning semiclassical pseudodifferential operators in the Weyl quantization, acting on quadratic exponentially weighted spaces of holomorphic functions (Bargmann spaces).

\medskip
\noindent
Let $\Phi_0$ be a strictly plurisubharmonic quadratic form on $\comp^n$. Associated to $\Phi_0$ we introduce the real $2n$-dimensional linear subspace
\begin{equation}
\label{psg.0}
\Lambda_{\Phi_0}=\left\{\left(x, \frac{2}{i}\frac{\partial\Phi_0}{\partial x}(x)\right), \,\,\, x\in \comp^n\right\} \subset \comp^{2n} = \comp^n_x \times \comp^n_{\xi}.
\end{equation}
The linear subspace $\Lambda_{\Phi_0}$ is I-Lagrangian and R-symplectic, in the sense that the restriction of the complex symplectic (2,0) form
\begeq
\label{psg.01}
\sigma = \sum_{j=1}^n d\xi_j \wedge dx_j
\endeq
on $\comp^n_x \times \comp^n_{\xi}$ to $\Lambda_{\Phi_0}$ is real and non-degenerate. In particular, $\Lambda_{\Phi_0}$ is maximally totally real. Let
\begin{equation}
\label{psg.0.5}
S(\Lambda_{\Phi_0}) = \big\{a\in C^\infty(\Lambda_{\Phi_0});\,\,\,
\partial^\alpha a=\mathcal{O}_\alpha(1)\quad \forall\alpha \in \nat^{2n}\big\}.
\end{equation}
We shall let symbols $a\in S(\Lambda_{\Phi_0})$ depend on the semiclassical parameter $h\in (0,1]$, provided that $a(\cdot;h)$ belongs to a bounded set in $S(\Lambda_{\Phi_0})$, when $h$ varies in $(0,1]$.

\bigskip
\noindent
Let $a\in S(\Lambda_{\Phi_0})$ and let $u\in\textrm{Hol}(\comp^n)$ be such that $u(x)=\mathcal{O}_{h,N}(1) \langle x \rangle^{-N} e^{\frac{\Phi_0(x)}{h}}$, for all $N\geq 0$. We set
\begin{equation}
\label{psg.1}
a^{w}_\Gamma(x,hD_x)u(x):=\frac{1}{(2\pi h)^n}\int\hskip-2mm\int_{\Gamma(x)} e^{\frac{i}{h}(x-y)\cdot \theta}
a\left(\frac{x+y}{2},\theta;h\right) u(y)\, dy \wedge\, d\theta,
\end{equation}
where $\Gamma(x)\subset \comp^n_{y,\theta}$ is the natural $2n$--dimensional contour of integration given by
\begeq
\label{psg.1.1}
\theta=\frac{2}{i}\frac{\partial \Phi_0}{\partial x}\left(\frac{x+y}{2}\right), \quad y\in \comp^n.
\endeq
The following consequence of Taylor's formula,
\begeq
\label{psg.1.1.1}
{\rm Re}\, \biggl(2 \partial_x \Phi_0\left(\frac{x+y}{2}\right) \cdot (x-y) \biggr) = \Phi_0(x) - \Phi_0(y),
\endeq
valid for the real valued quadratic form $\Phi_0$, assures that the integral in (\ref{psg.1}) converges absolutely. Let us also recall from~\cite{Sj96},~\cite{HiSj15} that $a^{w}_\Gamma(x,hD_x)u \in {\rm Hol}(\comp^n)$.

\bigskip
\noindent
It is established in~\cite{Sj96},~\cite{HiSj15} that the operator $a^w_{\Gamma}(x,hD_x)$ extends to a uniformly bounded map
\begeq
\label{psg.1.2}
a^w_{\Gamma}(x,hD_x) = {\cal O}(1): H_{\Phi_0}(\comp^n) \rightarrow H_{\Phi_0}(\comp^n).
\endeq
Here $H_{\Phi_0}(\comp^n)$ is the Bargmann space defined by
\begeq
\label{psg.1.3}
H_{\Phi_0}(\comp^n) =\textrm{Hol}(\comp^n)\cap L^2(\comp^n, e^{-2\Phi_0/h}L(dx)),
\endeq
with $L(dx)$ being the Lebesgue measure on $\comp^n$. The proof of the mapping property (\ref{psg.1.2}) given in~\cite{Sj96},~\cite{HiSj15} proceeds by introducing an almost holomorphic extension $\widetilde{a}\in C^{\infty}(\comp^{2n})$ of $a\in S(\Lambda_{\Phi_0})$, such that
$\partial^{\alpha} \widetilde{a}\in L^{\infty}(\comp^{2n})$ for all $\alpha$, and with the property ${\rm supp}\, \widetilde{a} \subset \Lambda_{\Phi_0} + B_{{\bf C}^{2n}}(0,\widetilde{C})$, for some $\widetilde{C}>0$. One then performs a contour deformation argument, letting $\Gamma^t(x) \subset \comp^{2n}_{y,\theta}$, $t\in [0,1]$, be the contour given by
\begin{equation}
\label{psg.2}
\theta=\frac{2}{i}\frac{\partial \Phi_0}{\partial x}\left(\frac{x+y}{2}\right) + i t \overline{(x-y)},\,\,\, \quad y\in \comp^n,
\end{equation}
and using Stokes' formula to get
\begeq
\label{psg.2.1}
a^{w}_{\Gamma}(x,hD_x)u = \widetilde{a}^{w}_{\Gamma^1}(x,hD_x)u + R u.
\endeq
Here
$$
\widetilde{a}^{w}_{\Gamma^1}(x,hD_x)u(x)=\frac{1}{(2\pi h)^n}\int\hskip-2mm\int_{\Gamma^1(x)} e^{\frac{i}{h}(x-y)\cdot \theta}
\widetilde{a}\left(\frac{x+y}{2},\theta;h\right) u(y)\, dy \wedge\, d\theta,
$$
and writing
\begeq
\label{psg.2.2}
\widetilde{a}^{w}_{\Gamma^1}(x,hD_x)u(x) = \int k_{\Gamma^1}(x,y;h) u(y)\,L(dy),
\endeq
we see, using (\ref{psg.1.1.1}) and (\ref{psg.2}), that the effective kernel $e^{-\frac{\Phi_0(x)}{h}}  k_{\Gamma^1}(x,y;h) e^{\frac{\Phi_0(y)}{h}}$ of the operator $\widetilde{a}^{w}_{\Gamma^1}(x,hD_x)$ satisfies
\begin{equation}
\label{psg.2.5}
e^{-\frac{\Phi_0(x)}{h}}  k_{\Gamma^1}(x,y;h) e^{\frac{\Phi_0(y)}{h}} = \mathcal{O}(1)h^{-n} e^{-\frac{1}{h}|x-y|^2}.
\end{equation}
Therefore, by Schur's lemma, we get
\begeq
\label{psg.2.6}
\widetilde{a}^{w}_{\Gamma^1}(x,hD_x) = {\cal O}(1): L^2(\comp^n, e^{-2\Phi_0/h}L(dx)) \rightarrow L^2(\comp^n, e^{-2\Phi_0/h}L(dx)),
\endeq
and it is shown in \cite[Proposition 1.2]{Sj96},~\cite[Section 1.4.]{HiSj15} that the remainder $R$ in (\ref{psg.2.1}) satisfies
\begeq
\label{psg.2.7}
R = {\cal O}(h^{\infty}): H_{\Phi_0}(\comp^n) \rightarrow L^2(\comp^n, e^{-2\Phi_0/h}L(dx)).
\endeq

\bigskip
\noindent
We are now ready to start the discussion of the Gevrey case. When doing so, let us notice first that identifying $\Lambda_{\Phi_0}$ linearly with $\comp^n_x$, via the projection map $\Lambda_{\Phi_0}\ni (x,\xi) \mapsto x\in \comp^n_x$, we may define the Gevrey spaces ${\cal G}^s(\Lambda_{\Phi_0})$, ${\cal G}^s_0(\Lambda_{\Phi_0})$, and ${\cal G}^s_b(\Lambda_{\Phi_0})$, for $s>1$.

\medskip
\noindent
Given $a\in {\cal G}^s_b(\Lambda_{\Phi_0})\subset S(\Lambda_{\Phi_0})$, for some $s>1$, we would like to establish an analogue of (\ref{psg.2.1}), (\ref{psg.2.6}), (\ref{psg.2.7}), replacing the deformed contour $\Gamma^1(x)$ in (\ref{psg.2}) by another one, if necessary, where we expect the Gevrey smoothness of $a$ to allow us to strengthen (\ref{psg.2.7}) to the estimate
\begeq
\label{psg.2.8}
R = {\cal O}(1)\,\exp\left(-\frac{1}{\mathcal{O}(1)}h^{-\frac{1}{s}}\right): H_{\Phi_0}(\comp^n) \rightarrow L^2(\comp^n, e^{-2\Phi_0/h}L(dx)).
\endeq
Specifically, we would like the effective kernel of the remainder to be
\begeq
\label{psg.2.9}
\mathcal{O}(h^{-n})\, \exp\left(-\frac{1}{\mathcal{O}(1)}h^{-\frac{1}{s}}\right).
\endeq
The motivation for such a decay estimate, as $h\rightarrow 0^+$, comes from the characterization of the space ${\cal G}^s_0(\real^d)$ via decay properties of the Fourier transforms, see Subsection \ref{Ftf}.

\bigskip
\noindent
Let $\widetilde{a}\in {\cal G}^s_b(\comp^{2n})$ be an almost holomorphic extension of $a\in {\cal G}^s_b(\Lambda_{\Phi_0})$, such that ${\rm supp}\, \widetilde{a} \subset \Lambda_{\Phi_0} + B_{{\bf C}^{2n}}(0,\widetilde{C})$, for some $\widetilde{C}>0$. The existence of such an extension has been established in Section \ref{ahol}, and we have the following natural analogue of \eqref{ahol.3},
\begin{equation}
\label{psg.4.7}
\abs{\overline{\partial} \widetilde{a}(x,\xi)} \leq
\mathcal{O}(1)\, \exp\left(-\frac{1}{\mathcal{O}(1)}\textrm{dist}\big((x,\xi),\Lambda_{\Phi_0}\big)^{-\frac{1}{s-1}}\right),\quad (x,\xi)\in \comp^{2n}.
\end{equation}
Proceeding similarly to the $C^{\infty}$ case, let us first perform a contour deformation to the contour $\Gamma^1(x)$ given in (\ref{psg.2}). We then still have (\ref{psg.2.1}), (\ref{psg.2.6}), and we only need to take a closer look at the remainder $R$ in (\ref{psg.2.1}), making use of the full strength of (\ref{psg.4.7}).

\medskip
\noindent
Stokes' formula gives that
\begeq
\label{psg.4.7.0}
Ru(x) = \frac{1}{(2\pi h)^n}\int\!\!\!\int\!\!\!\int_{G_{[0,1]}(x)}
e^{\frac{i}{h}(x-y)\cdot \theta}  u(y)\, \overline{\partial}\left(\widetilde{a}\left(\frac{x+y}{2},\theta\right)\right)\wedge \, dy \wedge\, d\theta,
\endeq
where $G_{[0,1]}(x)\subset \comp^n_y \times \comp^n_{\theta}$ is the $(2n+1)-$dimensional contour, given by \eqref{psg.2}, parametrized by $(t,y)\in [0,1]\times \comp^n$. Along $G_{[0,1]}(x)$, we have for $1\leq j \leq n$,
\begin{multline}
d\theta_j = \frac{2}{i}\sum_{k=1}^{n} \Phi''_{0,x_jx_k} \frac{1}{2} dy_k
+\frac{2}{i}\sum_{k=1}^{n} \Phi''_{0,x_j\overline{x_k}} \frac{1}{2} d\overline{y_k} -it d\overline{y_j} + i\overline{(x_j-y_j)} dt\\
= \sum_{k=1}^{n} {\cal O}(1) dy_k + \sum_{k=1}^{n} {\cal O}(1) d\overline{y_k} + i\overline{(x_j-y_j)} dt,
\end{multline}
and when computing $\overline{\partial}\left(\widetilde{a}(\frac{x+y}{2},\theta)\right)\wedge\, dy \wedge\, d\theta$, all the terms have to contain precisely one factor of $dt$. This form can therefore be expressed as $|x-y|{\cal O}(1) L(dy)dt$, and using (\ref{psg.1.1.1}), (\ref{psg.2}), and (\ref{psg.4.7}), we see that the expression in \eqref{psg.4.7.0} takes the form
\begeq
\label{psg.4.7.01}
Ru(x) = {\cal O}(1)\,h^{-n}
\int_0^1 dt \int_{{\bf C}^n} e^{\frac{1}{h} \left(\Phi_0(x) - \Phi_0(y) - t\abs{x-y}^2\right)} e^{-C_1\big(t\abs{x-y}\big)^{-\frac{1}{s-1}}} |x-y| u(y)\, L(dy),
\endeq
for some $C_1 > 0$. Here we have also used that along $G_{[0,1]}(x)$, we have
\begin{equation*}
\textrm{dist}\left(\Big(\frac{x+y}{2},\theta\Big),\Lambda_{\Phi_0}\right)=\mathcal{O}(1)t\abs{x-y}.
\end{equation*}
Writing
$$
Ru(x) = \int r(x,y;h) u(y)\, L(dy),
$$
we obtain that the effective kernel $e^{-\frac{\Phi_0(x)}{h}}  r(x,y;h) e^{\frac{\Phi_0(y)}{h}}$ of the operator $R$ in (\ref{psg.4.7.01}) satisfies
\begin{multline}
\label{psg.4.7.02}
e^{-\frac{\Phi_0(x)}{h}}  r(x,y;h) e^{\frac{\Phi_0(y)}{h}} = {\cal O}(1)\, h^{-n} \int_0^1
e^{- C_1(t\abs{x-y})^{-\frac{1}{s-1}}} e^{-\frac{t}{h} \abs{x-y}^2} |x-y|\, dt \\
\leq {\cal O}(1)\, h^{-n} \int_0^1 t^{-1/2} e^{- C_1(t\abs{x-y})^{-\frac{1}{s-1}}} e^{-\frac{t}{2h} \abs{x-y}^2}\, dt \\
\leq {\cal O}(1)\, h^{-n} \sup_{t\in[0,1]}
\left(\exp\left(-\frac{C}{h}(t\abs{x-y})^2 - C_1(t\abs{x-y})^{-\frac{1}{s-1}}\right)\right).
\end{multline}
Here $C$, $C_1>0$.

\medskip
\noindent
Setting
\begin{equation}
\label{psg.4.7.1}
g(\sigma)=\frac{C}{h}\sigma^2+C_1 \sigma^{-\frac{1}{s-1}},\quad{\mathrm{for}}\,\,\,\sigma >0,
\end{equation}
we can rewrite (\ref{psg.4.7.02}) as follows,
\begeq
\label{psg.4.7.1.1}
e^{-\frac{\Phi_0(x)}{h}}  r(x,y;h) e^{\frac{\Phi_0(y)}{h}} \leq \mathcal{O}(1)\, h^{-n}\,
\exp\left(-\inf_{\sigma>0} g(\sigma)\right).
\endeq
A straightforward computation shows that the infimum of $g$ over the positive half axis is attained at the unique point
\begin{equation}
\label{psg.4.7.2}
\sigma_{\textrm{min}}  = \left(\frac{C_1 h}{2C (s-1)}\right)^{\frac{s-1}{2s-1}},
\end{equation}
and the corresponding value of the infimum is equal to
\begeq
\label{psg.4.7.2.1}
\inf_{\sigma>0}g(\sigma)= \frac{C}{h}\sigma_{\min}^2+C_1 \sigma_{\min}^{-\frac{1}{s-1}}
=\frac{1}{\mathcal{O}(1)}h^{-\frac{1}{2s-1}}.
\endeq
We get, using (\ref{psg.4.7.1.1}) and (\ref{psg.4.7.2.1}),
\begin{equation}
\label{psg.4.7.3}
e^{-\frac{\Phi_0(x)}{h}}  r(x,y;h) e^{\frac{\Phi_0(y)}{h}} \leq \mathcal{O}(h^{-n})\, \exp\left(-\frac{1}{\mathcal{O}(1)}h^{-\frac{1}{2s-1}}\right),
\end{equation}
which is a strictly larger upper bound that the desired one in (\ref{psg.2.9}), for all $s>1$. We may therefore regard the discussion above as an indication of the fact that the deformed contour $\Gamma^1(x)$ in (\ref{psg.2}), natural in the analytic case~\cite{Sj82},~\cite{Sj96},~\cite{HiSj15}, is not quite adapted to the Gevrey theory.

\bigskip
\noindent
As a new attempt, we shall now consider the following piecewise smooth Lipschitz "mixed" contour $\Gamma_{\omega}(x) \subset \comp^{2n}_{y,\theta}$, defined as follows,
\begin{equation}
\label{psg.4}
{\Gamma_\omega}(x): \quad \theta=\frac{2}{i}\frac{\partial \Phi_0}{\partial x}\left(\frac{x+y}{2}\right)
 + if_\omega(x-y),\quad y\in \comp^n,
\end{equation}
with
\begin{equation}
\label{psg.4.5}
f_\omega({z})= \begin{cases}
                           \hskip10pt\overline{{z}}, \quad \, |{z}|\leq \omega,\\
                           {}\\
                           \displaystyle \omega\frac{\overline{{z}}}{|{z}|}, \quad |{z}| > \omega.
                          \end{cases}
\end{equation}
Here $0 < \omega < \sigma_{\textrm{min}}$ is to be chosen, with $\sigma_{\textrm{min}}$ given in (\ref{psg.4.7.2}).

\medskip
\noindent
In view of (\ref{psg.1.1.1}), (\ref{psg.4}), and (\ref{psg.4.5}), we have along ${\Gamma_\omega}(x)$,
\begin{equation}
\label{psg.5}
\textrm{Re}\left(i(x-y)\cdot \theta\right) + \Phi_0(y)-\Phi_0(x) = -{\rm Re}\, \left((x-y)\cdot f_{\omega}(x-y)\right) = - F_{\omega}(x-y),
\end{equation}
where
\begin{equation}
\label{psg.5.5}
0 \leq F_\omega({z})= \begin{cases}
                           \hskip5pt\vert{{z}}\vert^2, \quad \, |{z}|\leq \omega, \\
                           {}\\
                           \displaystyle \omega\vert{z}\vert, \quad\, |{z}| > \omega.
                          \end{cases}
\end{equation}

\bigskip
\noindent
The $2n$--dimensional contours $\Gamma(x)$ in (\ref{psg.1.1}) and $\Gamma_{\omega}(x)$ in (\ref{psg.4}) are homotopic, with the homotopy given by the family of contours,
\begin{equation}
\label{psg.5.51}
{\Gamma_\omega}(x,t): \quad \theta=\frac{2}{i}\frac{\partial \Phi_0}{\partial x}\left(\frac{x+y}{2}\right)
 + it f_\omega(x-y), \quad y\in \comp^n,
 \end{equation}
for $t\in [0,1]$. Let also $G_{[0,1],\omega}(x) \subset \comp^{2n}_{y,\theta}$ be the $(2n+1)$--dimensional contour given by (\ref{psg.5.51}),
parametrized by $(t,y)\in [0,1]\times \comp^n$. When $u\in {\rm Hol}(\comp^n)$ is such that $u(x)=\mathcal{O}_{h,N}(1) \langle x \rangle^{-N} e^{\frac{\Phi_0(x)}{h}}$, for all $N\geq 0$, we have, similarly to (\ref{psg.2.1}), by an application of Stokes' formula,
\begeq
\label{psg.5.6}
a^{w}_{\Gamma}(x,hD_x)u = \widetilde{a}^{w}_{\Gamma_{\omega}}(x,hD_x)u + R u.
\endeq
Here
\begeq
\label{psg.5.7}
\widetilde{a}^{w}_{\Gamma_{\omega}}(x,hD_x)u(x)=\frac{1}{(2\pi h)^n}\int\hskip-2mm\int_{\Gamma_{\omega}(x)} e^{\frac{i}{h}(x-y)\cdot \theta}
\widetilde{a}\left(\frac{x+y}{2},\theta;h\right) u(y)\, dy \wedge\, d\theta,
\endeq
and
\begeq
\label{psg.5.8}
Ru(x) = \frac{1}{(2\pi h)^n}\int\!\!\!\int\!\!\!\int_{G_{[0,1],\omega}(x)}
e^{\frac{i}{h}(x-y)\cdot \theta}  u(y) \, \overline{\partial}\left(\widetilde{a}\left(\frac{x+y}{2},\theta\right)\right)\wedge \, dy \wedge\, d\theta.
\endeq
We shall now estimate the effective kernel of the operator $R$ in (\ref{psg.5.8}). When doing so, we notice that along $G_{[0,1],\omega}(x)$, we have in view of (\ref{psg.5.51}),
\begeq
\label{psg.5.9}
\textrm{dist}\left(\left(\frac{x+y}{2},\theta\right),\Lambda_{\Phi_0}\right) \leq {\cal O}(1) t \abs{f_{\omega}(x-y)},
\endeq
and using also (\ref{psg.1.1.1}), (\ref{psg.5}), (\ref{psg.4.7}), and (\ref{psg.5.9}), we conclude that, similarly to (\ref{psg.4.7.01}), we can write
\begeq
\label{psg.5.9.1}
Ru(x) = {\cal O}(1)\,h^{-n}
\int_0^1 dt \int_{{\bf C}^n} e^{\frac{1}{h} \left(\Phi_0(x) - \Phi_0(y) - tF_{\omega}(x-y)\right)}
e^{-C_1\big(t\abs{f_{\omega}(x-y)}\big)^{-\frac{1}{s-1}}} u(y)\, L(dy),
\endeq
for some $C_1 > 0$. Setting
$$
Ru(x) = \int r(x,y;h)u(y)\, L(dy),
$$
we obtain from (\ref{psg.5.9.1}), (\ref{psg.4.5}), and (\ref{psg.5.5}) that the absolute value of the effective kernel $e^{-\Phi_0(x)/h} r(x,y;h) e^{\Phi_0(y)/h}$ of the operator $R$ in (\ref{psg.5.8}) does not exceed
\begin{equation}
\label{kernel}
\begin{aligned}
\mathcal{O}(1)\, h^{-n}\, \sup_{t\in [0,1]}
\begin{cases}
\exp\left(-\frac{C}{h}(t|x-y|)^2 - C_1(t\,\vert x-y\vert)^{-\frac{1}{s-1}}\right),
\,\,\,|x-y|\leq \omega,\,\,\, \\
{}\\
\exp\left(-\frac{C}{h}\omega\, t\,  |x-y| - C_1(t\, \omega)^{-\frac{1}{s-1}}\right),
\qquad \,\,\,\,|x-y| > \omega.\,\,
\end{cases}
\end{aligned}
\end{equation}
 Here $C$, $C_1>0$.

\bigskip
\noindent
We shall now discuss the choice of the parameter $0 < \omega < \sigma_{\rm min}$ in (\ref{psg.4}), (\ref{psg.4.5}), and here our goal is to achieve an upper bound of the form (\ref{psg.2.9}) for (\ref{kernel}). First, in the region $\abs{x-y}\leq \omega$, we have in view of (\ref{kernel}),
$$
e^{-\Phi_0(x)/h} r(x,y;h) e^{\Phi_0(y)/h} \leq \mathcal{O}(1) h^{-n}\, \exp\left(-\inf_{t\in [0,1]} g(t\abs{x-y})\right),
$$
where the function $g$ has been defined in (\ref{psg.4.7.1}). Since $g$ is decreasing on the interval $(0,\sigma_{\rm min})$, it suffices to choose $0 < \omega < \sigma_{\rm min}$ so that
\begeq
\label{psg.5.9.2}
g(\omega)\ge \frac{1}{\mathcal{O}(1)}h^{-\frac{1}{s}}.
\endeq
To this end, recalling (\ref{psg.4.7.1}), let us choose $0 < \omega$ such that $\omega^{-\frac{1}{s-1}}=\displaystyle \frac{1}{\mathcal{O}(1)} h^{-\frac{1}{s}}$, i.e.,
\begin{equation}
\label{psg.7}
\omega = \frac{1}{C_0} h^{1-\frac{1}{s}} \ll \sigma_{\rm min}.
\end{equation}
The choice (\ref{psg.7}) assures that (\ref{psg.5.9.2}) holds, and let us also notice that the first term in the expression for $g(\omega)$ satisfies
$$
\frac{C}{h} \omega^2 \ll h^{-\frac{1}{s}},
$$
so that
$$
g(\omega) \asymp h^{-\frac{1}{s}}.
$$
Here and in what follows we write $A\asymp B$ for $A,B\in \real$ if $A,\, B$ have the same sign (or vanish), and we have $A=\mathcal{O}(B)$ and $B=\mathcal{O}(A)$.

\medskip
\noindent
We conclude therefore that in the region $\abs{x-y}\leq \omega$, with $\omega$ given in (\ref{psg.7}), we have
\begin{multline}
\label{psg.5.9.3}
e^{-\Phi_0(x)/h} r(x,y;h) e^{\Phi_0(y)/h}
\leq \mathcal{O}(1) h^{-n}\, \exp\left(-g(\abs{x-y})\right) \\
\leq \mathcal{O}(1) h^{-n}\, \exp\left(-\frac{1}{{\cal O}(1)} h^{-\frac{1}{s}}\right).
\end{multline}
Next, a straightforward computation shows that in the region
\begeq
\label{psg.5.9.4}
\omega < \abs{x-y} \leq \frac{C_1}{C(s-1)} C_0^{\frac{s}{s-1}},
\endeq
the function
\begeq
\label{psg.5.9.5}
[0,1]\ni t \mapsto \frac{C}{h}\omega t\abs{x-y} + C_1 \left(t\omega\right)^{-\frac{1}{s-1}}
\endeq
is decreasing, and therefore using (\ref{kernel}) we obtain in the region (\ref{psg.5.9.4}),
\begin{multline}
\label{psg.5.9.6}
e^{-\Phi_0(x)/h} r(x,y;h) e^{\Phi_0(y)/h}
\leq \mathcal{O}(1) h^{-n}\, \exp\left(-\frac{C}{h}\omega\abs{x-y} - C_1 \omega^{-\frac{1}{s-1}}\right) \\
\leq \mathcal{O}(1) h^{-n}\, \exp\left(-\frac{1}{{\cal O}(1)} h^{-\frac{1}{s}}\right).
\end{multline}
Here we have also used (\ref{psg.7}). Finally, in the exterior region
\begeq
\label{psg.6.91}
\frac{C_1}{C(s-1)} C_0^{\frac{s}{s-1}} < \abs{x-y},
\endeq
the function in (\ref{psg.5.9.5}) achieves its infimum at the unique critical point
$$
t_{\rm min} = \left(\frac{C_1}{C(s-1)}\right)^{\frac{s-1}{s}} \frac{C_0}{\abs{x-y}^{(s-1)/s}} \in (0,1),
$$
and the corresponding critical value is of the form
\begeq
\label{psg.6.92}
\frac{1}{{\cal O}(1)} h^{-\frac{1}{s}} \abs{x-y}^{1/s}.
\endeq
We get therefore in the region (\ref{psg.6.91}),
\begin{multline}
\label{psg.6.93}
e^{-\Phi_0(x)/h} r(x,y;h) e^{\Phi_0(y)/h}
\leq \mathcal{O}(1) h^{-n}\, \exp\left(-\frac{1}{{\cal O}(1)} h^{-\frac{1}{s}} \abs{x-y}^{1/s}\right) \\
\leq \mathcal{O}(1) h^{-n}\, \exp\left(-\frac{1}{{\cal O}(1)} h^{-\frac{1}{s}}\right).
\end{multline}

\medskip
\noindent
Combining (\ref{psg.5.9.3}), (\ref{psg.5.9.6}), and (\ref{psg.6.93}), we conclude that the effective kernel of the operator $R$ in (\ref{psg.5.8}) obeys an upper bound of the form (\ref{psg.2.9}), provided that $\omega$ is chosen as in (\ref{psg.7}).

\bigskip
\noindent
It is now easy to derive precise bounds on the operator norms of the operators in (\ref{psg.5.7}) and (\ref{psg.5.8}), viewed as linear continuous maps on the $L^2$--space $L^2(\comp^n, e^{-2\Phi_0/h}L(dx))$. Indeed, an application of Schur's lemma together with (\ref{psg.5.9.3}), (\ref{psg.5.9.6}), and (\ref{psg.6.93}), shows first that the operator norm of $R$ in (\ref{psg.5.8}) does not exceed
\begin{multline}
\label{psg.6.94}
{\cal O}(1) h^{-n} \int_{\abs{x}\leq \omega} \exp(-g(\abs{x}))\, L(dx) + {\cal O}(1) h^{-n} \int_{\omega \leq \abs{x}\leq {\cal O}(1)}
\exp\left(- C_1 \omega^{-\frac{1}{s-1}}\right)\,L(dx) \\
+ {\cal O}(1) h^{-n} \int_{{\cal O}(1)\leq \abs{x}} \exp\left(-\frac{1}{{\cal O}(1)} h^{-\frac{1}{s}} \abs{x}^{1/s}\right)\, L(dx) = I_1 + I_2 + I_3,
\end{multline}
with the function $g$ defined in (\ref{psg.4.7.1}). Here we clearly have
$$
I_j = \mathcal{O}(1)\,\exp\left(-\frac{1}{\mathcal{O}(1)} h^{-\frac{1}{s}}\right),\quad j = 2,3,
$$
in view of (\ref{psg.7}), and when estimating the first contribution in (\ref{psg.6.94}), we obtain in view of (\ref{psg.5.9.3}),
$$
I_1 \leq {\cal O}(1) h^{-n} \exp\left(-\frac{1}{\mathcal{O}(1)} h^{-\frac{1}{s}}\right) \leq \mathcal{O}(1)\,
\exp\left(-\frac{1}{2\mathcal{O}(1)} h^{-\frac{1}{s}}\right).
$$
We get therefore,
\begeq
\label{psg.6.95}
R = \mathcal{O}(1)\,\exp\left(-\frac{1}{\mathcal{O}(1)} h^{-\frac{1}{s}}\right): L^2(\comp^n, e^{-2\Phi_0/h}L(dx)) \rightarrow L^2(\comp^n, e^{-2\Phi_0/h}L(dx)).
\endeq

\medskip
\noindent
Next, turning the attention to the operator $\widetilde{a}^w_{\Gamma_{\omega}}(x,hD_x)$ in (\ref{psg.5.7}), and writing
$$
\widetilde{a}^w_{\Gamma_{\omega}}(x,hD_x) u(x) = \int k_{\Gamma_{\omega}}(x,y;h)\, L(dx),
$$
we get in view of (\ref{psg.5}),
\begeq
\label{psg.6.96}
e^{-\Phi_0(x)/h} k_{\Gamma_{\omega}}(x,y;h) e^{\Phi_0(y)/h} \leq {\cal O}(1) h^{-n} \exp\left(-F_{\omega}(x-y)/h\right).
\endeq
Recalling (\ref{psg.5.5}), in view of Schur's lemma, we only have to control the $L^1$ norm
\begin{multline}
\label{psg.6.97}
{\cal O}(1) h^{-n} \int \exp\left(-\frac{F_{\omega}(x)}{h}\right)\, L(dx) \\
\leq {\cal O}(1) h^{-n} \int_{\abs{x}\leq \omega} \exp\left(-\frac{\abs{x}^2}{h}\right)\, L(dx) + {\cal O}(1) h^{-n}
\int_{\abs{x}\geq \omega} \exp\left(-\frac{\omega \abs{x}}{h}\right)\, L(dx) \\
= {\cal O}(1) + {\cal O}(1) \frac{h^n}{\omega^{2n}} = {\cal O}(1) + {\cal O}(1) h^{-n\left(1-\frac{2}{s}\right)}.
\end{multline}
We conclude that
\begeq
\label{psg.6.98}
\widetilde{a}^w_{\Gamma_{\omega}}(x,hD_x) = {\cal O}(1) {\rm max}\, \left(1, h^{-n\left(1-\frac{2}{s}\right)}\right): L^2(\comp^n, e^{-2\Phi_0/h}L(dx)) \rightarrow L^2(\comp^n, e^{-2\Phi_0/h}L(dx)),
\endeq
and in particular, this operator is ${\cal O}(1)$ precisely when $1< s\leq 2$.

\bigskip
\noindent
We may summarize the discussion above in the following theorem, which is the main result of this subsection.
\begin{theo}
\label{main}
Let $\Phi_0$ be a strictly plurisubharmonic quadratic form on $\comp^n$ and let $a\in {\cal G}^s_b(\Lambda_{\Phi_0})$, for some $s>1$. Let $\widetilde{a} \in {\cal G}^s_b(\comp^{2n})$ be an almost holomorphic extension of $a$ such that ${\rm supp}\, \widetilde{a} \subset \Lambda_{\Phi_0} + B_{{\bf C}^{2n}}(0,C)$, for some $C>0$, or more generally, let $\widetilde{a} \in C^1_b(\comp^{2n})$ be an extension of $a$ with the same support properties, such that {\rm (\ref{psg.4.7})} holds. Let furthermore $\Gamma_{\omega}(x) \subset \comp^{2n}_{y,\theta}$ be the piecewise smooth Lipschitz contour given in {\rm (\ref{psg.4})}, {\rm (\ref{psg.4.5})}, where $0 < \omega$ satisfies {\rm (\ref{psg.7})}. We have
\begin{equation}
\label{psg.10}
a^{w}_{\Gamma}(x,hD_x) = \widetilde{a}^{w}_{{\Gamma_\omega}}(x,hD_x) + R,
\end{equation}
where the operator $\widetilde{a}^w_{{\Gamma_\omega}}(x,hD_x)$ in {\rm (\ref{psg.5.7})} satisfies
\begeq
\label{psg.10.1}
\widetilde{a}^w_{\Gamma_{\omega}}(x,hD_x) = {\cal O}(1) {\rm max}\, \left(1, h^{-n\left(1-\frac{2}{s}\right)}\right): H_{\Phi_0}(\comp^n) \rightarrow L^2(\comp^n, e^{-2\Phi_0/h}L(dx)),
\endeq
and
\begeq
\label{psg.10.2}
R =\mathcal{O}(1)\,\exp\left(-\frac{1}{\mathcal{O}(1)} h^{-\frac{1}{s}}\right): L^2(\comp^n, e^{-2\Phi_0/h}L(dx)) \rightarrow L^2(\comp^n, e^{-2\Phi_0/h}L(dx)).
\endeq
\end{theo}

\medskip
\noindent
Recalling the approximate uniqueness of almost holomorphic extensions, see (\ref{ahol.4}), (\ref{uni.34}), we also get the following result.

\begin{coro}(Dependence on the choice of an almost holomorphic extension.)
\label{dce}
Let $\Phi_0$ be a strictly plurisubharmonic quadratic form on $\comp^n$ and let $a\in {\cal G}^s_b(\Lambda_{\Phi_0})$, for some $s>1$.
Let $\widetilde{a}_1$, $\widetilde{a}_2 \in C^{\infty}_b(\comp^{2n})$ be two almost holomorphic extensions of $a$, such that for $j=1,2$, we have ${\rm supp}\, \widetilde{a_j} \subset \Lambda_{\Phi_0} + B_{{\bf C}^{2n}}(0,C)$, for some $C>0$, and also,
\begin{equation}
\label{psg.10.3}
\abs{\partial^{\alpha} \overline{\partial}^{\beta} \overline{\partial} \widetilde{a}_j(\rho)} \leq
\mathcal{O}(1)\, \exp\left(-\frac{1}{\mathcal{O}(1)}{\rm{dist}}\big(\rho,\Lambda_{\Phi_0}\big)^{-\frac{1}{s-1}}\right),\quad \rho \in \comp^{2n},
\end{equation}
for $\abs{\alpha} + \abs{\beta}\leq 2$. Then we have
\begin{equation}
\label{psg.11}
(\widetilde{a}_1)^{w}_{{\Gamma_\omega}}(x,hD_x) = (\widetilde{a}_2)^{w}_{{\Gamma_\omega}}(x,hD_x) + R,
\end{equation}
where
$$
(\widetilde{a}_j)^{w}_{{\Gamma_\omega}}(x,hD_x) = \mathcal{O}(1)\max\left(1,h^{-n(1-\frac{2}{s})}\right): H_{\Phi_0}(\comp^n) \rightarrow L^2(\comp^n, e^{-2\Phi_0/h}L(dx)),
$$
for $j=1,2$, and
$$
R = \mathcal{O}(1)\, \exp\left(-\frac{1}{\mathcal{O}(1)} h^{-\frac{1}{s}}\right): L^2(\comp^n, e^{-2\Phi_0/h}L(dx)) \rightarrow L^2(\comp^n, e^{-2\Phi_0/h}L(dx)).
$$
\end{coro}

\bigskip
\noindent
{\it Remark}. We have, for $(y,\theta)\in \Gamma_{\omega}(x)$, in view of (\ref{psg.4}), (\ref{psg.4.5}),
$$
{\rm dist}\,\left(\left(\frac{x+y}{2},\theta\right),\Lambda_{\Phi_0}\right) \leq \omega = \frac{1}{C_0} h^{1 - \frac{1}{s}}.
$$
When working in the Gevrey category, we should therefore stay closer to the real domain $\Lambda_{\Phi_0}$, than in the analytic case, see (\ref{psg.2}).

\medskip
\noindent
{\it Remark}. As we saw in the beginning of this subsection, using analytic contours, such as $\Gamma^1(x)$ in (\ref{psg.2}), leads to estimates of the form (\ref{psg.4.7.3}) for the effective kernels of the remainders, that are not quite precise. Closely related to this observation is the phenomenon of the loss of Gevrey smoothness in stationary phase expansions, see~\cite{LaLa97},~\cite{Rouleux}. To illustrate it in a simple setting, let $a\in {\cal G}^s_0(\real^d)$, for some $s>1$. Arguing as in~\cite[Exercise 2.4]{GrSj}, we see that
there exists $C>0$ such that for all $N\in \nat$, we have
\begeq
\label{psg.11.01}
\frac{1}{(2\pi h)^{d/2}} \int e^{-\abs{x}^2/2h}\, a(x)\, dx = \sum_{j=0}^{N-1} \frac{h^j}{j!} \left(\frac{\Delta}{2}\right)^j a(0) + R_N(h),
\endeq
where
\begeq
\label{psg.11.02}
\abs{R_N(h)} \leq C^{N+1}(N!)^{2s-1} h^N.
\endeq
Choosing $N \sim (1/Ch)^{1/(2s-1)}$ leads to the remainder estimate of the form
$$
{\cal O}(1)\, \exp\left(-\frac{1}{{\cal O}(1)} h^{-\frac{1}{2s-1}}\right).
$$

\bigskip
\noindent
We shall finish this subsection by discussing the dependence of the realization $\widetilde{a}^w_{\Gamma_{\omega}}$
on the choice of the parameter $\omega$, such that $\omega \asymp \displaystyle h^{1 - \frac{1}{s}}$. To this end, let $0< \omega_j$, $j=1,2$, be such that
\begeq
\label{psg.11.1}
\omega_j \asymp h^{1-\frac{1}{s}}, \quad j=1,2,
\endeq
and let us introduce the natural homotopy between the contours $\Gamma_{\omega_1}$ and $\Gamma_{\omega_2}$, given by
\begin{equation*}
{\Gamma}_{\omega_t}(x):\quad \theta = \frac{2}{i}\frac{\partial \Phi_0}{\partial x}\left(\frac{x+y}{2}\right)
 + i\,f_{\omega_{t}}(x-y), \quad y\in \comp^n,
\end{equation*}
with $\omega_t=(1-t)\omega_1+t\omega_2$, $t\in[0,1]$. Introducing also the $(2n+1)$--dimensional contour $\underset{t\in[0,1]}\bigcup\Gamma_{\omega_t}(x)$ and applying Stokes' formula, we get
\begeq
\label{psg.12}
\widetilde{a}^w_{\Gamma_{\omega_1}}(x,hD_x)u = \widetilde{a}^w_{\Gamma_{\omega_2}}(x,hD_x)u + Ru.
\endeq
Here, similarly to (\ref{psg.5.9.1}), the remainder $R$ takes the form
\begeq
\label{psg.12.1}
Ru(x) = {\cal O}(1)\,h^{-n}
\int_0^1 dt \int_{{\bf C}^n} e^{\frac{1}{h} \left(\Phi_0(x) - \Phi_0(y) - F_{\omega_t}(x-y)\right)}
e^{-C_1 \abs{f_{\omega_t}(x-y)}^{-\frac{1}{s-1}}} u(y)\, L(dy),
\endeq
for some $C_1 > 0$. To control the norm of the operator
$$
R: L^2(\comp^n, e^{-2\Phi_0/h}L(dx)) \rightarrow L^2(\comp^n, e^{-2\Phi_0/h}L(dx)),
$$
in (\ref{psg.12.1}), it suffices, in view of Schur's lemma, to estimate the $L^1$ norm
\begeq
\label{psg.12.2}
{\cal O}(1)\,h^{-n} \int_0^1 dt \int e^{-\frac{F_{\omega_t}(x)}{h} - C_1 \abs{f_{\omega_t}(x)}^{-\frac{1}{s-1}}}\, L(dx) = I_1 + I_2.
\endeq
Here
\begeq
\label{psg.12.3}
I_1 = \mathcal{O}(1)h^{-n}\int_0^1 dt \int_{\abs{x} \leq \omega_t} \exp\left(-\frac{|x|^2}{h} - C_1 \abs{x}^{-\frac{1}{s-1}}\right)\, L(dx),
\endeq
and
\begeq
\label{psg.12.4}
I_2 = \mathcal{O}(1)h^{-n}\int_0^1 dt \int_{\omega_t < \abs{x}} \exp\left(-\frac{\abs{x}\omega_t}{h} - C_1 \omega_t^{-\frac{1}{s-1}}\right)\,L(dx).
\endeq
We have
\begin{multline*}
I_1 =  \mathcal{O}(1)h^{-n}\,\int_0^1 dt\, \exp\left(-C_1 \omega_t^{-\frac{1}{s-1}}\right)\, \int_{\abs{x} \leq \omega_t} \exp\left(-\frac{\abs{x}^2}{h}\right)\, L(dx) \\
\leq {\cal O}(1) \int_0^1 \exp\left(-C_1 \omega_t^{-\frac{1}{s-1}}\right)\, dt \leq {\cal O}(1)\, \exp\left(-\frac{1}{{\cal O}(1)} h^{-\frac{1}{s}}\right),
\end{multline*}
since $\omega_t^{-1/(s-1)} \geq \displaystyle \frac{1}{{\cal O}(1)} h^{-1/s}$, for $0\leq t \leq 1$. Furthermore, making the change of variables
$y = x\omega_t/h$ in (\ref{psg.12.4}), we get
\begeq
I_2 \leq \mathcal{O}(1)h^{n} \int_0^1 \omega_t^{-2n} \exp\left(- C_1 \omega_t^{-\frac{1}{s-1}}\right)\, dt \leq {\cal O}(1)\, \exp\left(-\frac{1}{{\cal O}(1)} h^{-\frac{1}{s}}\right).
\endeq
We conclude that
\begin{multline}
\label{psg.12.5}
\widetilde{a}^w_{\Gamma_{\omega_1}}(x,hD_x) - \widetilde{a}^w_{\Gamma_{\omega_2}}(x,hD_x) \\ = \mathcal{O}(1)\, \exp\left(-\frac{1}{\mathcal{O}(1)}h^{-\frac{1}{s}}\right): L^2(\comp^n, e^{-2\Phi_0/h}L(dx)) \rightarrow L^2(\comp^n, e^{-2\Phi_0/h}L(dx)),
\end{multline}
provided that $0 < \omega_j$ are such that (\ref{psg.11.1}) holds.

\subsection{Deformations of exponential weights}
\label{chw}
\noindent
Let $a\in {\cal G}^s_b(\Lambda_{\Phi_0})$, where $s>1$, and let $\widetilde{a} \in {\cal G}^s_b(\comp^{2n})$ be an almost holomorphic extension of $a$ such that ${\rm supp}\, \widetilde{a} \subset \Lambda_{\Phi_0} + B_{{\bf C}^{2n}}(0,C)$, for some $C>0$. In Theorem \ref{main}, it was established that
\begeq
\label{chw.01}
\widetilde{a}^w_{\Gamma_{\omega}}(x,hD_x) = {\cal O}(1) \max\left(1,h^{-n(1-\frac{2}{s})}\right): H_{\Phi_0}(\comp^n) \rightarrow L^2(\comp^n, e^{-2\Phi_0/h}L(dx)).
\endeq
Here the $\Gamma^{\Phi_0}_{\omega}(x) := \Gamma_{\omega}(x) \subset \comp^{2n}_{y,\theta}$ has been introduced in (\ref{psg.4}), (\ref{psg.4.5}), with the parameter $0 < \omega$ given in (\ref{psg.7}). The Gevrey smoothness of $a$ allows us to consider deformations of the quadratic weight function $\Phi_0$ --- see~\cite{Sj82},~\cite{Sj96},~\cite{MeSj01} for this idea in the analytic case, where $\displaystyle \frac{1}{{\cal O}(1)}$--perturbations of $\Phi_0$ are allowed, and~\cite{DeSjZw04},~\cite{HeSjSt} for the $C^{\infty}$--theory, where deformations should be ${\cal O}(h \abs{\log h})$--close to $\Phi_0$. See also~\cite[Chapter 3]{M_book} for the Gevrey case.

\bigskip
\noindent
Let $\Phi_1 = \Phi_0 + \psi \in C^{1,1}(\comp^n;\real)$ be such that
\begin{equation}
\label{chw.1}
\norm{\nabla^k \psi}_{L^{\infty}({\bf C}^n)} \leq \frac{\omega}{{\cal O}(1)},\quad k = 0,1,2,
\end{equation}
where the implicit constant in (\ref{chw.1}) is large enough, and let $\Gamma^{\Phi_1}_{\omega}(x) \subset \comp^{2n}_{y,\theta}$ be the following Lipschitz contour adapted to the weight $\Phi_1$, defined analogously to (\ref{psg.4}),
\begeq
\label{chw.1.1}
\Gamma^{\Phi_1}_{\omega}(x): \quad \theta = \frac{2}{i}\frac{\partial \Phi_1}{\partial x}\left(\frac{x+y}{2}\right) + i f_{\omega}(x-y),\quad y\in \comp^n.
\endeq
Here $f_{\omega}$ has been defined in (\ref{psg.4.5}), and $0 < \omega$ satisfies (\ref{psg.7}). We would like to replace the contour $\Gamma^{\Phi_0}_{\omega}(x)$ in (\ref{psg.5.7}) by $\Gamma^{\Phi_1}_{\omega}(x)$, and to this end we introduce the natural intermediate family of contours,
\begin{equation}
\label{chw.1.2}
{\Gamma}^{\Phi_{t}}_\omega(x): \quad \theta = \frac{2}{i}\frac{\partial \Phi_{t}}{\partial x}\left(\frac{x+y}{2}\right)
 +if_{\omega}(x-y), \quad y\in \comp^n,
\end{equation}
where $\Phi_{t}:=(1-t)\Phi_0+t\Phi_1$, $t\in [0,1]$. Let $G_{[0,1]}(x)\subset \comp^{2n}_{y,\theta}$ be the $(2n+1)$--dimensional contour given by (\ref{chw.1.2}), parametrized by $(t,y)\in [0,1]\times \comp^n$, and write, by an application of Stokes' formula,
\begin{multline}
\label{chw.2}
\widetilde{a}^w_{{\Gamma^{\Phi_0}_{\omega}}}(x,hD_x)u(x) - \widetilde{a}^w_{{\Gamma^{\Phi_1}_{\omega}}}(x,hD_x)u(x) \\
= \frac{1}{(2\pi h)^n}\int\!\!\!\int\!\!\!\int_{G_{[0,1]}(x)}
e^{\frac{i}{h}(x-y)\cdot \theta} u(y)\overline{\partial}\left(\widetilde{a}\left(\frac{x+y}{2},\theta\right)\right) \wedge\, dy \wedge \, d\theta =: Ru(x).
\end{multline}
Here $u\in H_{\Phi_0}(\comp^n)$. Along $G_{[0,1]}(x)$, we have in view of (\ref{chw.1}), (\ref{chw.1.2}),
\begeq
\label{chw.2.1}
\textrm{dist}\left(\left(\frac{x+y}{2},\theta\right),\Lambda_{\Phi_0}\right) \leq {\cal O}(1)t \abs{\nabla \psi\left(\frac{x+y}{2}\right)} + \abs{f_{\omega}(x-y)} \leq {\cal O}(1)\,\omega,
\endeq
and combining (\ref{chw.2.1}) with (\ref{psg.4.7}), (\ref{psg.5}), (\ref{chw.1}), and (\ref{chw.1.2}), we get
\begin{multline}
\label{chw.2.2}
Ru(x) \\ = {\cal O}(1)\,h^{-n}
\int_0^1 dt \int e^{\frac{1}{h} \left(\Phi_0(x) - \Phi_0(y) - F_{\omega}(x-y) + t\omega\abs{x-y}/{\cal O}(1)\right)}
\exp\left(-\frac{1}{{\cal O}(1)}h^{-\frac{1}{s}}\right) u(y)\, L(dy).
\end{multline}
The absolute value of the effective kernel of the operator $R$ in (\ref{chw.2.1}) does not exceed therefore
$$
{\cal O}(1)\,h^{-n} \exp\left(-\frac{1}{{\cal O}(1)}h^{-\frac{1}{s}}\right)\,\exp\left(\frac{1}{h} \left(- F_{\omega}(x-y) + \frac{\omega\abs{x-y}}{{\cal O}(1)}\right)\right).
$$
Recalling (\ref{psg.5.5}) and making use of Schur's lemma, we conclude, in view of (\ref{chw.2}), that
\begin{multline}
\label{chw.2.3}
\widetilde{a}^w_{{\Gamma^{\Phi_0}_{\omega}}}(x,hD_x)- \widetilde{a}^w_{{\Gamma^{\Phi_1}_{\omega}}}(x,hD_x) = \mathcal{O}(1)\,
\exp\left(-\frac{1}{\mathcal{O}(1)}h^{-\frac{1}{s}}\right): \\
H_{\Phi_0}(\comp^n) \rightarrow L^2(\comp^n,e^{-2\Phi_0/h} L(dx)).
\end{multline}
In view of (\ref{chw.01}), we have now established the first part of the following result.

\bigskip
\noindent
\begin{theo}
\label{perturbationresult}
Let ${\Phi}_1 = \Phi_0 + \psi \in C^{1,1}(\comp^n;\real)$ be such that {\rm (\ref{chw.1})} holds and let us introduce the contour $\Gamma_{\omega}^{\Phi_1}(x) \subset \comp^{2n}_{y,\theta}$, defined in {\rm (\ref{chw.1.1})}. The realization
$$
\widetilde{a}^w_{\Gamma^{\Phi_1}_\omega}(x,hD_x)u(x)=\frac{1}{(2\pi h)^n}\int\!\!\!\int_{\Gamma^{\Phi_1}_{\omega}(x)} e^{\frac{i}{h}(x-y)\cdot \theta}
\widetilde{a}\left(\frac{x+y}{2},\theta;h\right) u(y)\, dy \wedge\, d\theta
$$
enjoys the following mapping properties:
\begin{itemize}
\item[(i)] We have
\begeq
\label{chw.2.3.1}
\widetilde{a}^w_{\Gamma^{\Phi_1}_\omega}(x,hD_x) = \mathcal{O}(1) \max\left(1,h^{-n(1-\frac{2}{s})}\right)
: H_{\Phi_0}(\comp^n) \rightarrow L^2(\comp^n,e^{-2\Phi_0/h} L(dx)).
\endeq
\item[(ii)] We have
\begeq
\label{chw.2.3.2}
\widetilde{a}^w_{\Gamma^{\Phi_1}_\omega}(x,hD_x) = \mathcal{O}(1) \max\left(1,h^{-n(1-\frac{2}{s})}\right)
: H_{\Phi_1}(\comp^n) \rightarrow L^2(\comp^n,e^{-2\Phi_1/h} L(dx)).
\endeq
Here we have set $H_{\Phi_1}(\comp^n) = \mathrm{Hol}(\comp^n)\cap L^2(\comp^n,e^{-2\Phi_1/h} L(dx))$.
\end{itemize}
\end{theo}
\begin{proof}
We only need to check the validity of the second statement, and when doing so, let us consider along $\Gamma_{\omega}^{\Phi_1}(x)$,
\begin{multline}
\label{chw.2.4}
-\Phi_1(x)+ {\rm Re}\,\left(i(x-y)\cdot \theta\right) + \Phi_1(y) \\
= -\Phi_1(x) + {\rm Re}\,\left(2\frac{\partial\Phi_1}{\partial x}\Big(\frac{x+y}{2}\Big)\cdot(x-y)\right) +\Phi_1(y) - F_\omega(x-y) \\
= -\psi(x) + \biggl\langle{\nabla \psi\left(\frac{x+y}{2}\right),x-y\biggr\rangle}_{{\bf R}^{2n}} + \psi(y) - F_\omega(x-y).
\end{multline}
Here we have used (\ref{psg.1.1.1}) on the last line. We have
\begeq
\label{chw.2.5}
-\psi(x) + \biggl\langle{\nabla \psi\left(\frac{x+y}{2}\right),x-y\biggr\rangle}_{{\bf R}^{2n}} + \psi(y) \leq 2 \norm{\nabla \psi}_{L^{\infty}({\bf C}^n)}\abs{x-y},
\endeq
and an application of Taylor's formula gives that
\begin{multline}
\label{chw.2.6}
-\psi(x) + \biggl\langle{\nabla \psi\left(\frac{x+y}{2}\right),x-y\biggr\rangle}_{{\bf R}^{2n}} + \psi(y) \\
= \int_0^1 (1-t) \psi''\left(\frac{x+y}{2} - t\left(\frac{x-y}{2}\right)\right)\,\frac{(x-y)}{2}\cdot \frac{(x-y)}{2}\,dt \\
- \int_0^1 (1-t) \psi''\left(\frac{x+y}{2} + t\left(\frac{x-y}{2}\right)\right)\,\frac{(x-y)}{2}\cdot \frac{(x-y)}{2}\,dt \\
\leq \frac{1}{4} \norm{\nabla^2 \psi}_{L^{\infty}({\bf C}^n)}\abs{x-y}^2.
\end{multline}
Here the Hessian and the scalar product are taken in the sense of $\real^{2n}$. Writing
$$
\widetilde{a}^w_{\Gamma^{\Phi_1}_\omega}(x,hD_x)u(x) = \int k_{\Gamma^{\Phi_1}_\omega}(x,y;h)u(y) \, L(dy),
$$
we obtain, in view of (\ref{chw.1}), (\ref{chw.2.4}), (\ref{chw.2.5}), and (\ref{chw.2.6}), that the effective kernel of $\widetilde{a}^w_{\Gamma^{\Phi_1}_\omega}(x,hD_x)$ satisfies
\begin{multline}
\label{chw.2.7}
e^{-\frac{\Phi_1(x)}{h}} k_{\Gamma^{\Phi_1}_\omega}(x,y;h) e^{\frac{\Phi_1(y)}{h}} \\ \leq
{\cal O}(1) h^{-n}\, \exp\left(\frac{1}{h}\left(-F_{\omega}(x-y) + \frac{\omega \abs{x-y}}{{\cal O}(1)}{\rm min}\,\left(1,\abs{x-y}\right)\right)\right) \\ \leq {\cal O}(1) h^{-n}\, \exp\left(-\frac{1}{2h}F_{\omega}(x-y)\right),
\end{multline}
provided that the implicit constant in (\ref{chw.1}) is sufficiently large. The pointwise estimate (\ref{chw.2.7}), on the level of effective kernels, is therefore of the same kind as (\ref{psg.6.96}), and arguing as in the proof of Theorem \ref{main}, we get the operator norm bound (\ref{chw.2.3.2}).
\end{proof}

\bigskip
\noindent
Combining Theorem \ref{main} and Theorem \ref{perturbationresult}, we get
\begeq
\label{chw.2.8}
a^w_{\Gamma}(x,hD_x) = \widetilde{a}^w_{\Gamma^{\Phi_1}_{\omega}}(x,hD_x) + R,
\endeq
where
$$
R = \mathcal{O}(1)\,\exp\left(-\frac{1}{C}h^{-\frac{1}{s}}\right): \\
H_{\Phi_0}(\comp^n) \rightarrow L^2(\comp^n,e^{-2\Phi_0/h} L(dx)),
$$
for some $C>0$, and therefore
\begeq
\label{chw.2.9}
R = \mathcal{O}(1)\,\exp\left(-\frac{1}{2C}h^{-\frac{1}{s}}\right): \\
H_{\Phi_1}(\comp^n) \rightarrow L^2(\comp^n,e^{-2\Phi_1/h} L(dx)),
\endeq
provided that the implicit constant in (\ref{chw.1}) is large enough. Another application of Theorem \ref{perturbationresult} together with (\ref{chw.2.8}), (\ref{chw.2.9}) allows us to conclude that the operator $a^w_{\Gamma}(x,hD_x)$ extends to a uniformly bounded map
\begeq
\label{chw.2.10}
a^w_{\Gamma}(x,hD_x) = {\cal O}(1): H_{\Phi_1}(\comp^n) \rightarrow H_{\Phi_1}(\comp^n),
\endeq
for $1 < s \leq 2$, and we can view the operator $\widetilde{a}^w_{\Gamma^{\Phi_1}_{\omega}}(x,hD_x)$ as the correspon\-ding u\-ni\-formly bounded realization.

\bigskip
\noindent
In the remainder of this subsection, we shall be concerned with the problem of finding uniformly bounded realizations of the operator $a^w_{\Gamma}(x,hD_x)$ in the region $s>2$. As we shall see, we shall then have to accept a remainder which is larger than the one in (\ref{chw.2.9}). Let us start with the following largely heuristic remark.

\bigskip
\noindent
{\it Remark}. In Theorem \ref{main}, we have established that
\begin{multline*}
a^w_{\Gamma}(x,hD_x)-\widetilde{a}^w_{{\Gamma_\omega}}(x,hD_x)=\mathcal{O}(1)\,\exp\left(-\frac{1}{\mathcal{O}(1)} h^{-\frac{1}{s}}\right)
:\\ L^2(\comp^n,e^{-2\Phi_0/h} L(dx)) \rightarrow L^2(\comp^n,e^{-2\Phi_0/h} L(dx)),
\end{multline*}
where the realization $\widetilde{a}^w_{{\Gamma_\omega}}(x,hD_x)$ is uniformly bounded on $L^2(\comp^n,e^{-2\Phi_0/h} L(dx))$ in the range
$1 < s \leq 2$, while we only have
$$
\widetilde{a}^w_{{\Gamma_\omega}}(x,hD_x)= {\cal O}(1) h^{-n\left(1-\frac{2}{s}\right)}\,:\,L^2(\comp^n,e^{-2\Phi_0/h} L(dx)) \rightarrow L^2(\comp^n,e^{-2\Phi_0/h} L(dx)),
$$
for $s>2$. Our purpose here is to address the question whether there exists a (Lipschitz) contour $\widetilde\Gamma(x) \subset \comp^{2n}_{y,\theta}$
of dimension $2n$, such that the following two properties,
\begin{multline}
\label{chw.2.11}
a^w_{\Gamma}(x,hD_x)-\widetilde{a}^w_{\widetilde{\Gamma}}(x,hD_x)=\mathcal{O}(1)\,\exp\left(-\frac{1}{\mathcal{O}(1)} h^{-\frac{1}{s}}\right)
:\\ L^2(\comp^n,e^{-2\Phi_0/h} L(dx)) \rightarrow L^2(\comp^n,e^{-2\Phi_0/h} L(dx)),
\end{multline}
and
\begeq
\label{chw.2.12}
\widetilde{a}^w_{{\widetilde\Gamma}}(x,hD_x)=\mathcal{O}(1)\,:\,L^2(\comp^n,e^{-2\Phi_0/h} L(dx)) \rightarrow L^2(\comp^n,e^{-2\Phi_0/h} L(dx)),
\endeq
hold, for $s>2$. Indeed, let us pose the following question.

\medskip
\noindent
{\it Question:} Let $s>2$. Is there a (Lipschitz) function $f: \comp^n \rightarrow \comp^n$ such that with the choice
$$
\widetilde{\Gamma}(x): \quad \theta = \frac{2}{i} \frac{\partial \Phi_0}{\partial x}\left(\frac{x+y}{2}\right) + if(x-y),\quad y\in \comp^n,
$$
the properties (\ref{chw.2.11}), (\ref{chw.2.12}) hold?

\medskip
\noindent
The following discussion seems to indicate that the answer to the question is likely to be negative. Let us try $f$ of the form
$$
f(z)=\widehat{f}(\abs{z})\frac{\overline{z}}{\abs{z}},
$$
for a suitable continuous $\widehat{f}\geq 0$ on $[0,\infty)$. The absolute value of the effective kernel of the realization $\widetilde{a}^w_{{\widetilde\Gamma}}(x,hD_x)$ then does not exceed
\begeq
\label{chw.2.13}
{\cal O}(1) h^{-n}\, \exp\left(-\frac{1}{h}\abs{x-y} \widehat{f}(\abs{x-y})\right),
\endeq
and in view of Schur's lemma, the property (\ref{chw.2.12}) holds provided that
\begeq
\label{chw.2.14}
h^{-n}\, \int e^{-\frac{1}{h}\abs{x} \widehat{f}(\abs{x})}\, L(dx) \leq {\cal O}(1).
\endeq
On the other hand, introducing the intermediate contours given by
$$
\widetilde{\Gamma}_t(x): \quad \theta = \frac{2}{i} \frac{\partial \Phi_0}{\partial x}\left(\frac{x+y}{2}\right) + itf(x-y),\quad y\in \comp^n,
$$
for $t\in [0,1]$, and applying Stokes' formula, we see that the effective kernel of the contribution coming from the region
$\underset{0\leq t\leq 1}{\bigcup}\widetilde{\Gamma}_t(x)$ has the form
\begeq
\label{chw.2.15}
{\cal O}(1) h^{-n} \int_0^1 \exp\left(-\frac{1}{h}t\widehat{f}(\abs{x-y})\abs{x-y} - C_1 \left(t\widehat{f}(\abs{x-y})\right)^{-\frac{1}{s-1}}\right)\, dt,
\endeq
for some $C_1 >0$. Here we have ignored the possible trouble coming from the Jacobian ${\rm det}\,(\partial_{\overline{y}}\theta)$. Recalling (\ref{chw.2.11}) and taking $C_1 = 1$ in (\ref{chw.2.15}) for simplicity, we are led to the following pointwise condition on $\widehat{f}$,
$$
\frac{1}{h}\widehat{f}(r)r+\widehat{f}(r)^{-\frac{1}{s-1}}\geq \frac{h^{-\frac{1}{s}}}{\mathcal{O}(1)},\quad r\geq 0.
$$
In particular, we need that for each $r\geq 0$, uniformly,
$$
\frac{1}{h}\widehat{f}(r)r\geq \frac{h^{-\frac{1}{s}}}{\mathcal{O}(1)}\quad {\textrm{or}}\quad
\widehat{f}(r)^{-\frac{1}{s-1}}\geq  \frac{h^{-\frac{1}{s}}}{\mathcal{O}(1)},
$$
or equivalently,
$$
\widehat{f}(r)\geq \frac{h^{1-\frac{1}{s}}}{\mathcal{O}(1)r}\quad {\textrm{or}}\quad
\widehat{f}(r)\leq \frac{h^{1-\frac{1}{s}}}{\mathcal{O}(1)}.
$$
Using that $\widehat{f}$ is bounded near $0$, we conclude that $\widehat{f}(r)\leq  \frac{h^{1-\frac{1}{s}}}{\mathcal{O}(1)}$
on some non-trivial interval of the form $[0,\frac{1}{\mathcal{O}(1)}]$. In other words,
$$
\widehat{f}(r) \leq \mathcal{O}(1) \omega, \quad \textrm{for}\quad 0\leq r\leq \frac{1}{\mathcal{O}(1)},
$$
and therefore we get
$$
h^{-n}\, \int e^{-\frac{1}{h}\abs{x} \widehat{f}(\abs{x})}\, L(dx) \geq h^{-n} \int_{\abs{x}\leq \frac{1}{{\cal O}(1)}} e^{-\frac{{\cal O}(1)}{h}\omega \abs{x}}\, L(dx) \asymp  h^{-n\left(1-\frac{2}{s}\right)}.
$$
Here the expression in the right hand side is unbounded as $h\rightarrow 0^+$, for $s>2$, which is incompatible with (\ref{chw.2.14}).

\bigskip
\noindent
When finding a uniformly bounded realization of the operator $a^w_{\Gamma}(x,hD_x)$ on the space $H_{\Phi_1}(\comp^n)$, for $s>2$, we are going to perform an additional contour deformation, starting from the unbounded realization given by $\widetilde{a}^w_{\Gamma^{\Phi_1}_{\omega}}(x,hD_x)$. The price that we have to pay is that we should then allow for a remainder that is larger than the one in (\ref{chw.2.9}), and is only moderately smaller than the remainder naturally associated to the contour (\ref{psg.2}) used in the analytic theory, see \eqref{psg.4.7.3}.

\begin{theo}
\label{contour2}
Assume that $s>2$, and let $\Phi_1 = \Phi_0 + \psi\in C^{1,1}(\comp^n;\real)$ be such that {\rm (\ref{chw.1})} holds. Let $\Gamma^{\Phi_1}_{h^{1/2}}(x) \subset \comp^{2n}_{y,\theta}$ be the contour, defined as in {\rm (\ref{chw.1.1})}, with $\omega$ replaced by $h^{1/2}$. We have
\begeq
\label{chw.3}
\widetilde{a}^w_{\Gamma^{\Phi_1}_\omega}(x,hD_x) = \widetilde{a}^w_{\Gamma^{\Phi_1}_{h^{1/2}}}(x,hD_x)+ R,
\endeq
where the realization
$$
\widetilde{a}^w_{\Gamma^{\Phi_1}_{h^{1/2}}}(x,hD_x)u(x)=\frac{1}{(2\pi h)^n}\int\!\!\!\int_{\Gamma^{\Phi_1}_{h^{1/2}}(x)} e^{\frac{i}{h}(x-y)\cdot \theta}
\widetilde{a}\left(\frac{x+y}{2},\theta;h\right) u(y)\, dy \wedge\, d\theta
$$
satisfies
\begeq
\label{chw.3.1}
\widetilde{a}^w_{\Gamma^{\Phi_1}_{h^{1/2}}}(x,hD_x) = {\cal O}(1): H_{\Phi_1}(\comp^n) \rightarrow L^2(\comp^n,e^{-2\Phi_1/h} L(dx)).
\endeq
Furthermore,
\begeq
\label{chw.3.2}
R = \mathcal{O}(1)\,\exp\left(-\frac{1}{{\cal O}(1)}h^{-\frac{1}{2s-2}}\right): \\
H_{\Phi_1}(\comp^n) \rightarrow L^2(\comp^n,e^{-2\Phi_1/h} L(dx)).
\endeq
\end{theo}
\begin{proof}
Let us note, first of all, that
\begeq
\label{chw.3.3}
\omega=\frac{1}{\mathcal{O}(1)}h^{1 -\frac{1}{s}} \leq h^{1/2},\qquad \textrm{for all}\quad s>2.
\endeq
With this in mind, we shall adapt the approach used at the end of subsection \ref{cds}. Let us introduce the natural family of intermediate contours,
\begin{equation}
\label{chw.3.4}
\Gamma^{\Phi_1}_{\omega_{t}}(x): \quad \theta = \frac{2}{i}\frac{\partial \Phi_1}{\partial x}\left(\frac{x+y}{2}\right)
 + i\,f_{\omega_{t}}(x-y),\quad y\in \comp^n,
\end{equation}
where $\omega_t=(1-t)\omega + th^{1/2}$, $t\in[0,1]$. An application of Stokes' formula gives that
\begin{multline}
\label{chw.3.5}
\widetilde{a}^w_{{\Gamma^{\Phi_1}_{\omega}}}(x,hD_x)u(x) - \widetilde{a}^w_{{\Gamma^{\Phi_1}_{h^{1/2}}}}(x,hD_x)u(x) \\
= \frac{1}{(2\pi h)^n}\int\!\!\!\int\!\!\!\int_{\underset{t\in[0,1]}\bigcup\Gamma^{\Phi_1}_{\omega_{t}}(x)}
e^{\frac{i}{h}(x-y)\cdot \theta} u(y)\overline{\partial}\left(\widetilde{a}\left(\frac{x+y}{2},\theta\right)\right) \wedge\, dy \wedge \, d\theta =: Ru(x).
\end{multline}
As usual, let us now proceed to estimate the effective kernel of the operator $R$ in (\ref{chw.3.5}). To this end, we notice that along
the contour $\underset{t\in[0,1]}\bigcup\Gamma^{\Phi_1}_{\omega_{t}}(x)$ we have, similarly to (\ref{chw.2.4}),
\begin{multline}
\label{chw.3.6}
-\Phi_1(x)+ {\rm Re}\,\left(i(x-y)\cdot \theta\right) + \Phi_1(y) \\
= -\psi(x) + \biggl\langle{\nabla \psi\left(\frac{x+y}{2}\right),x-y\biggr\rangle}_{{\bf R}^{2n}} + \psi(y) - F_{\omega_t}(x-y) \\
\leq - F_{\omega_t}(x-y) + \frac{\omega_t\abs{x-y}}{{\cal O}(1)}{\rm min}(1,\abs{x-y}) \leq -\frac{1}{2} F_{\omega_t}(x-y).
\end{multline}
Here we have also used (\ref{chw.2.5}), (\ref{chw.2.6}), as well as the fact that $\omega \leq \omega_t$, for $t\in [0,1]$. Noticing also that along the contour $\underset{t\in[0,1]}\bigcup\Gamma^{\Phi_1}_{\omega_{t}}(x)$ we have, in view of (\ref{chw.3.3}) and (\ref{chw.3.4}),
$$
\mathrm{dist}\left(\left(\frac{x+y}{2},\theta\right),\Lambda_{\Phi_0}\right) \leq 2\abs{\nabla \psi\left(\frac{x+y}{2}\right)} + \omega_t
\leq \frac{\omega}{{\cal O}(1)} + \omega_t \leq 2 h^{1/2},
$$
we conclude that the absolute value of the effective kernel of the operator $R$ in (\ref{chw.3.5}), for the boundedness on $L^2(\comp^n, e^{-2\Phi_1/h}L(dx))$, does not exceed
\begeq
\label{chw.3.7}
{\cal O}(1)\,h^{-n} \int_0^1 e^{-\frac{1}{2h} F_{\omega_t}(x-y)}\, \exp\left(-\frac{1}{{\cal O}(1)}h^{-\frac{1}{2(s-1)}}\right) dt.
\endeq
In view of Schur's lemma and (\ref{chw.3.7}), to estimate the operator norm of $R$, we have to control the $L^1$--norm
\begeq
\label{chw.3.8}
{\cal O}(1)\,h^{-n} \int_0^1 dt\, \int e^{-\frac{1}{2h} F_{\omega_t}(x)}\, \exp\left(-\frac{1}{{\cal O}(1)}h^{-\frac{1}{2(s-1)}}\right) L(dx) = I_1 + I_2,
\endeq
where
\begeq
\label{chw.3.9}
I_1 = \mathcal{O}(1)h^{-n}\int_0^1 dt\, \int_{|x|\leq \omega_t} \exp\left(-\frac{|x|^2}{\mathcal{O}(1)h}\right)
\exp\left(-\frac{1}{{\cal O}(1)}h^{-\frac{1}{2(s-1)}}\right) L(dx),
\endeq
and
\begeq
I_2 = \mathcal{O}(1)h^{-n}\int_0^1 dt\, \int_{\omega_t<\abs{x}}
\exp\left(-\frac{|x|\omega_t}{\mathcal{O}(1)h}\right) \exp\left(-\frac{1}{{\cal O}(1)}h^{-\frac{1}{2(s-1)}}\right) L(dx).
\endeq
We have
\begeq
\label{chw.3.10}
I_1 = \mathcal{O}(1)\, \exp\left(-\frac{1}{{\cal O}(1)}h^{-\frac{1}{2(s-1)}}\right),
\endeq
and
\begin{multline}
\label{chw.3.11}
I_2 = \mathcal{O}(1)\, \exp\left(-\frac{1}{{\cal O}(1)}h^{-\frac{1}{2(s-1)}}\right)\int_0^1 \frac{h^n}{\omega_t^{2n}}\, dt \\
\leq \mathcal{O}(1)\, \exp\left(-\frac{1}{{\cal O}(1)}h^{-\frac{1}{2(s-1)}}\right) \frac{h^n}{\omega^{2n}} \leq \mathcal{O}(1)\,
\exp\left(-\frac{1}{2{\cal O}(1)}h^{-\frac{1}{2(s-1)}}\right).
\end{multline}
The estimate (\ref{chw.3.2}) follows, in view of (\ref{chw.3.5}), (\ref{chw.3.8}), (\ref{chw.3.10}), and (\ref{chw.3.11}).

\bigskip
\noindent
We shall finally verify the uniform boundedness property (\ref{chw.3.1}) for the realization $\widetilde{a}^w_{\Gamma^{\Phi_1}_{h^{1/2}}}(x,hD_x)$. To this end, let us observe that along the contour $\Gamma^{\Phi_1}_{h^{1/2}}(x)$, we have, similarly to (\ref{chw.3.6}),
\begeq
\label{chw.3.12}
-\Phi_1(x)+ {\rm Re}\,\left(i(x-y)\cdot \theta\right) + \Phi_1(y) \leq -\frac{1}{2} F_{h^{1/2}}(x-y).
\endeq
Writing
$$
\widetilde{a}^w_{\Gamma^{\Phi_1}_{h^{1/2}}}(x,hD_x)u(x) = \int k_{\Gamma^{\Phi_1}_{h^{1/2}}}(x,y;h)u(y) \, L(dy),
$$
we obtain therefore, in view of (\ref{chw.3.12}),
\begeq
\label{chw.3.13}
e^{-\frac{\Phi_1(x)}{h}} k_{\Gamma^{\Phi_1}_{h^{1/2}}}(x,y;h) e^{\frac{\Phi_1(y)}{h}} \leq {\cal O}(1)\, h^{-n} \exp\left(-\frac{1}{2} F_{h^{1/2}}(x-y)\right).
\endeq
The pointwise bound (\ref{chw.3.13}), on the level of effective kernels, is therefore of the same kind as (\ref{psg.6.96}), with the only difference that the small parameter $\omega$ has been replaced by $h^{1/2} \geq \omega$. An application of Schur's lemma gives therefore immediately (\ref{chw.3.1}). This completes the proof.
\end{proof}

\bigskip
\noindent
Combining (\ref{chw.2.8}), (\ref{chw.2.9}), and Theorem \ref{contour2}, we get in the region $s>2$,
\begeq
\label{chw.3.14}
a^w_{\Gamma}(x,hD_x) = \widetilde{a}^w_{\Gamma^{\Phi_1}_{h^{1/2}}}(x,hD_x) + R,
\endeq
where (\ref{chw.3.1}) holds, and the remainder $R$ satisfies (\ref{chw.3.2}). We conclude in particular that the operator $a^w_{\Gamma}(x,hD_x)$ extends to a uniformly bounded map
\begeq
\label{chw.3.15}
a^w_{\Gamma}(x,hD_x) = {\cal O}(1): H_{\Phi_1}(\comp^n) \rightarrow H_{\Phi_1}(\comp^n),
\endeq
for $s>2$, and we can view the operator $\widetilde{a}^w_{\Gamma^{\Phi_1}_{h^{1/2}}}(x,hD_x)$ as the corresponding uniformly bounded realization.

\medskip
\noindent
Theorem \ref{theo_main1} and Theorem \ref{theo_main2} in the introduction now follow from Theorem \ref{main}, Theorem \ref{perturbationresult}, and Theorem \ref{contour2}.

\medskip
\noindent
{\it Remark}. In the work~\cite{HiLaSjZe20}, prepared simultaneously with the present one, the mapping property (\ref{chw.3.15}) in the range $s\geq 2$ is established using alternative methods, not relying upon the contour deformations techniques.

\medskip
\noindent
The discussion in this section gives, in particular, the following result.
\begin{coro}
Let $a\in {\cal G}^s_b(\Lambda_{\Phi_0})$, $s>1$, and let $\Phi_1 = \Phi_0 + \psi \in C^{1,1}(\comp^n;\real)$ be such that {\rm (\ref{chw.1})} holds. The operator ${\rm Op}_h^w(a)$ extends to a uniformly bounded map
$$
{\rm Op}_h^w(a) = {\cal O}(1): H_{\Phi_1}(\comp^n) \rightarrow H_{\Phi_1}(\comp^n).
$$
\end{coro}

\subsection{Phase symmetries and composition of Gevrey operators}
\label{Comp}
In the first part of this subsection, we shall develop an approach to the composition of semiclassical Weyl quantizations in the complex domain, based on the representation of the operators as superpositions of suitable phase symmetries~\cite{Se63}. Such an approach is carried out in~\cite{Le10} in the real setting, and here we shall adapt it to the present complex environment.

\medskip
\noindent
Let $\Phi_0$ be a strictly plurisubharmonic quadratic form on $\comp^n$ and let the I-Lagrangian R-symplectic linear subspace $\Lambda_{\Phi_0}\subset \comp^n_x \times \comp^n_{\xi}$ be given by (\ref{psg.0}). Given $a\in {\cal S}(\Lambda_{\Phi_0})$, let us consider following (\ref{psg.1}),
\begin{equation}
\label{comp.01}
a^w_{\Gamma}(x,hD_x)u(x)=\frac{1}{(2\pi h)^n}\int\!\!\!\int_{\Gamma(x)} e^{\frac{i}{h}(x-y)\cdot \theta} a\left(\frac{x+y}{2},\theta\right)u(y)\, dy\wedge \, d\theta.
\end{equation}
Here $u\in H_{\Phi_0}(\comp^n)$ and $\Gamma(x)\subset \comp^{2n}_{y,\theta}$ is the contour given by (\ref{psg.1.1}). Setting
\begeq
\label{comp.02}
a_{\Phi_0}(x) = a\left(x,\frac{2}{i}\frac{\partial \Phi_0}{\partial x}(x)\right), \quad x\in \comp^n,
\endeq
and recalling (\ref{psg.1.1.1}), we can write in view of (\ref{comp.01}),
\begin{multline}
\label{comp.03}
a^w_{\Gamma}(x,hD_x)u(x) \\ = \frac{2^n {\rm det}\,(\Phi''_{0,\overline{x}x})}{(2\pi h)^n}\int e^{\frac{1}{h}\left(\Phi_0(x) - \Phi_0(y) + 2i {\rm Im}\,\left((x-y)\cdot \partial_x \Phi_0\left(\frac{x+y}{2}\right)\right)\right)} a_{\Phi_0}\left(\frac{x+y}{2}\right) u(y)\, L(dy).
\end{multline}
Here we have also used that along $\Gamma(x)$, we have $dy\wedge d\theta = 2^n\, {\rm det}\,(\Phi''_{0,\overline{x}x}) L(dy)$, provided that the orientation has been chosen suitably.

\medskip
\noindent
Let $u,v\in H_{\Phi_0}(\comp^n)$, and let us set $U=e^{-\Phi_0/h}u \in L^2(\comp^n)$, $V=e^{-\Phi_0/h}v \in L^2(\comp^n)$. We get, using (\ref{comp.03}),
\begin{multline}
\label{comp.1}
\left(a^w_{\Gamma}(x,hD_x)u,v\right)_{H_{\Phi_0}} = \int a^w_{\Gamma}(x,hD)u(x)\, \overline{v(x)} e^{-\frac{2}{h}\Phi_0(x)}\, L(dx)\\
= \frac{2^n {\rm det}\,(\Phi''_{0,\overline{x}x})}{(2\pi h)^n} \int\!\!\!\int e^{\frac{2i}{h}{\rm Im}\,\left((x-y)\cdot \partial_x \Phi_0\left(\frac{x+y}{2}\right)\right)} a_{\Phi_0}\left(\frac{x+y}{2}\right) U(y) \overline{V(x)}\, L(dy)\, L(dx).
\end{multline}
Making the linear change of variables in (\ref{comp.1}),
$$
x'=\frac{x+y}{2}, \quad y' = x-y,
$$
where the absolute value of the Jacobian is $1$, we obtain after dropping the primes,
\begin{multline}
\label{comp.2}
\left(a^w_{\Gamma}(x,hD_x)u,v\right)_{H_{\Phi_0}} \\
= \frac{2^n {\rm det}\,(\Phi''_{0,\overline{x}x})}{(2\pi h)^n} \int\!\!\!\int e^{\frac{2i}{h}{\rm Im}\,\left(y\cdot \partial_x \Phi_0(x)\right)} a_{\Phi_0}(x) U\left(x - \frac{1}{2}y\right) \overline{V\left(x + \frac{1}{2}y\right)}\, L(dy)\, L(dx), \\
= \frac{2^n {\rm det}\,(\Phi''_{0,\overline{x}x})}{(2\pi h)^n} \int a_{\Phi_0}(x)\mathcal{K}(U,V)(x)\, L(dx).
\end{multline}
Here ${\cal K}(U,V)(x)$ is "the Wigner function" given by
\begin{equation}
\label{comp.3}
{\cal K}(U,V)(x)=\int e^{\frac{2i}{h}{\rm Im}\,\left(y\cdot \partial_x \Phi_0(x)\right)} U\left(x-\frac{1}{2}y\right)\, \overline{V\left(x+\frac{1}{2}y\right)}\, L(dy).
\end{equation}
Performing the change of variables $y \mapsto \widetilde{y} = x +\frac{1}{2}y$ in (\ref{comp.3}), we get after dropping the tilde,
\begin{multline}
\label{comp.4}
{\cal K}(U,V)(x) = 2^{2n} \int e^{\frac{4i}{h}{\rm Im}\,\left((y-x)\cdot \partial_x \Phi_0(x)\right)} U(2x-y) \overline{V(y)}\,L(dy) \\
 = 2^{2n}(\Sigma_{x}U,V)_{L^2({\bf C}^n)},
\end{multline}
where $\Sigma_x$, $x\in \comp^n$, is the unitary map on $L^2(\comp^n)$ given by
\begin{equation}
\label{comp.5}
(\Sigma_{x}U)(y) = e^{\frac{4i}{h}{\rm Im}\,\left((y-x)\cdot \partial_x \Phi_0(x)\right)} U(2x-y).
\end{equation}
We obtain, combining (\ref{comp.2}) and (\ref{comp.4}),
\begin{equation}
\label{comp.6}
\left(a^w_{\Gamma}(x,hD_x)u,v\right)_{H_{\Phi_0}} = \frac{2^n {\rm det}\,(\Phi''_{0,\overline{x}x})}{(2\pi h)^n} \int a_{\Phi_0}(x) 2^{2n}(\Sigma_{x}U,V)_{L^2({\bf C}^n)}\, L(dx),
\end{equation}
and therefore,
\begin{equation}
\label{comp.7}
e^{-\frac{\Phi_0}{h}} a^w_{\Gamma}(x,hD_x) e^{\frac{\Phi_0}{h}}  = \frac{2^n {\rm det}\,(\Phi''_{0,\overline{x}x})}{(2\pi h)^n}
\int a_{\Phi_0}(x) 2^{2n}\Sigma_{x}\,L(dx).
\end{equation}
Here we may notice that the realization $a^w_{\Gamma}(x,hD_x)$ of the operator $a^w(x,hD_x)$ in (\ref{comp.01}) acts on the weighted $L^2$--space $L^2(\comp^n, e^{-2\Phi_0/h}L(dx))$, whereas $a^w(x,hD_x)$ is defined on the holomorphic subspace only. 
The decomposition (\ref{comp.7}) can be regarded as the complex analogue of the corresponding representation obtained in~\cite[Chapter 2]{Le10} in the real domain. When deriving an explicit formula for the composition $a^w(x,hD_x)\circ b^w(x,hD_x)$, for $a,b\in {\cal S}(\Lambda_{\Phi_0})$, we shall proceed by computing first the composition $\Sigma_{y}\circ \Sigma_{z}$ for $y,z\in \comp^n$.

\bigskip
\noindent
When doing so, let us consider the decomposition
\begin{equation}
\label{comp.8}
\Phi_0 = \Phi_{{\rm herm}} + \Phi_{\rm plh},
\end{equation}
where $\Phi_{{\rm herm}}(x) = \Phi''_{0,\overline{x}x}x\cdot \overline{x}$ is positive definite Hermitian and $\Phi_{{\rm plh}}(x) = {\rm Re}\, \left(\Phi''_{0,xx}x\cdot x\right)$ is pluriharmonic. Let
$$
A = \frac{2}{i} \left(\Phi_{{\rm plh}}\right)''_{xx} = \frac{2}{i} \Phi''_{0,xx}.
$$
The complex linear canonical transformation
\begeq
\label{comp.8.1}
\comp^{2n} \ni (y,\eta) \mapsto \kappa_A(y,\eta) = (y, \eta - Ay) \in \comp^{2n}
\endeq
satisfies
$$
\kappa_A\left(\Lambda_{\Phi_0}\right) = \Lambda_{\Phi_{{\rm herm}}},
$$
and associated to $\kappa_A$ is the metaplectic Fourier integral operator
\begeq
\label{comp.8.2}
{\cal U}u = ue^{-f}, \quad f(x) = \Phi''_{0,xx}x\cdot x,
\endeq
which maps $H_{\Phi_0}(\comp^n)$ unitarily onto $H_{\Phi_{{\rm herm}}}(\comp^n)$. By an application of the exact Egorov theorem we get
$$
{\cal U}\circ a^w(x,hD_x) \circ {\cal U}^{-1} = b^w(x,hD_x),
$$
where $b\in {\cal S}(\Lambda_{\Phi_{{\rm herm}}})$ is given by $b = a\circ \kappa_A^{-1}$. Conjugating $a^w(x,hD_x)$ by the operator ${\cal U}$ in (\ref{comp.8.2}), we obtain a reduction to the case when the pluriharmonic part of $\Phi_0$ vanishes, and in what follows, we shall therefore make this assumption.

\medskip
\noindent
The unitary map $\Sigma_x$ in (\ref{comp.5}) takes the form
\begin{multline}
\label{comp.8.3}
(\Sigma_x U)(y) = e^{\frac{4i}{h}{\rm Im}\,\left(\Phi''_{0,\overline{x}x}(y-x)\cdot \overline{x}\right)} U(2x-y) \\=
e^{\frac{4i}{h}{\rm Im}\,\left(\Phi''_{0,\overline{x}x}y \cdot \overline{x}\right)} U(2x-y) = e^{\frac{4i}{h}{\rm Im}\,\Psi_0(y,\overline{x})} U(2x-y),
\end{multline}
where $\Psi_0$ is the polarization of $\Phi_0$, i.e., the unique holomorphic quadratic form on $\comp^n_x \times \comp^n_y$ such that $\Psi_0(x,\overline{x}) = \Phi_0(x)$.

\begin{lemma}\label{compSym}
Let $\Phi_0$ be a strictly plurisubharmonic quadratic form on $\comp^n$ with vanishing pluriharmonic part. We have for $y,z\in \comp^n$,
\begin{equation}
\label{comp.10}
\Sigma_{y} \circ\Sigma_{z}= \frac{2^n {\rm det}\,(\Phi''_{0,\overline{x}x})}{(2\pi h)^n} \int e^{\frac{8i}{h}{\rm Im}\, \Psi_0\left(x-y,\overline{x}- \overline{z}\right)} 2^{2n} \Sigma_{x}\,L(dx).
\end{equation}
Here $\Psi_0$ is the polarization of $\Phi_0$.
\end{lemma}
\begin{proof}
By a direct computation, using (\ref{comp.8.3}), we get
\begin{equation}
\label{comp.11}
\left(\Sigma_{y}\circ\Sigma_{z} U\right)(y') = e^{\frac{4i}{h}{\rm Im}\, \Psi_0(y', \overline{y} - \overline{z})}
e^{\frac{8i}{h} {\rm Im}\, \Psi_0(y,\overline{z})}\, U(y'-2y+2z).
\end{equation}
On the other hand, the operator in the right hand side of (\ref{comp.10}) acting on $U$, is given by
\begin{multline}
\label{comp.12}
\left(\mathbb{L} U\right)(y'):= \frac{2^n {\rm det}\,(\Phi''_{0,\overline{x}x})}{(2\pi h)^n} \int e^{\frac{8i}{h}{\rm Im}\, \Psi_0\left(x-y,\overline{x}- \overline{z}\right)} 2^{2n} \left(\Sigma_{x}U\right)(y')\,L(dx) \\
= \frac{2^n {\rm det}\,(\Phi''_{0,\overline{x}x})}{(2\pi h)^n} e^{\frac{8i}{h} {\rm Im}\, \Psi_0(y,\overline{z})} \int 2^{2n} e^{\frac{4i}{h}{\rm Im}\, \Psi_0\left(y'-2y+2z,\overline{x}\right)} U(2x-y')\, L(dx).
\end{multline}
Here we have also used the skew-symmetry property ${\rm Im}\, \Psi_0(x,\overline{z}) = -{\rm Im}\, \Psi_0(z,\overline{x})$. Making the change of variables $\zeta =2x-y'$ in (\ref{comp.12}), we get
\begin{multline}
\label{comp.13}
\left(\mathbb{L} U\right)(y') = \frac{2^n {\rm det}\,(\Phi''_{0,\overline{x}x})}{(2\pi h)^n} e^{\frac{8i}{h} {\rm Im}\, \Psi_0(y,\overline{z})}
\int e^{\frac{4i}{h}{\rm Im}\, \Psi_0(y'-2y+2z,\frac{\overline{y'} + \overline{\zeta}}{2})} U(\zeta)\, L(d\zeta) \\
= \frac{2^n {\rm det}\,(\Phi''_{0,\overline{x}x})}{(2\pi h)^n} e^{\frac{8i}{h} {\rm Im}\, \Psi_0(y,\overline{z})}
e^{\frac{4i}{h} {\rm Im}\, \Psi_0(y',\overline{y} - \overline{z})}
\int e^{\frac{2i}{h}{\rm Im}\, \Psi_0(y'-2y+2z,\overline{\zeta})} U(\zeta)\, L(d\zeta).
\end{multline}
On the other hand, taking $a=1$ in (\ref{comp.7}), we obtain for $W\in L^2(\comp^n)$ such that $e^{\Phi_0/h}W\in H_{\Phi_0}(\comp^n)$,
\begin{multline}
\label{comp.14}
W(y') = \frac{2^n {\rm det}\,(\Phi''_{0,\overline{x}x})}{(2\pi h)^n} \int 2^{2n} e^{\frac{4i}{h}{\rm Im}\, \Psi_0(y',\overline{x})} W(2x-y')\,L(dx)\\
= \frac{2^n {\rm det}\,(\Phi''_{0,\overline{x}x})}{(2\pi h)^n} \int W(\zeta)e^{\frac{2i}{h}{\rm Im}\, \Psi_0(y',\overline{\zeta})} \,L(d\zeta).
\end{multline}
Here on the second line we have again made the change of variables $\zeta=2x-y'$. Using (\ref{comp.14}) we conclude that the expression in the right hand side of (\ref{comp.13}) becomes
$$
e^{\frac{8i}{h} {\rm Im}\, \Psi_0(y,\overline{z})}
e^{\frac{4i}{h} {\rm Im}\, \Psi_0(y',\overline{y} - \overline{z})} U(y'-2y+2z),
$$
which agrees with $\left(\Sigma_{y}\circ\Sigma_{z} U\right)(y')$, in view of (\ref{comp.11}). The proof is complete.
\end{proof}

\bigskip
\noindent
We are now ready to compute the composition of two Weyl quantizations. Let $a,b\in {\cal S}(\Lambda_{\Lambda_{\Phi_0}})$, and let us write following (\ref{comp.7}),
$$
e^{-\frac{\Phi_0}{h}} a^w_{\Gamma}(x,hD_x) e^{\frac{\Phi_0}{h}}  = \frac{{\rm det}\,(\Phi''_{0,\overline{x}x})}{(\pi h)^n}
\int a_{\Phi_0}(y) 2^{2n}\Sigma_{y}\,L(dy),
$$
$$
e^{-\frac{\Phi_0}{h}} b^w_{\Gamma}(x,hD_x) e^{\frac{\Phi_0}{h}}  = \frac{{\rm det}\,(\Phi''_{0,\overline{x}x})}{(\pi h)^n}
\int b_{\Phi_0}(z) 2^{2n}\Sigma_{z}\,L(dz).
$$
Using (\ref{comp.10}), we get
\begin{multline}
\label{comp.14.2}
e^{-\frac{\Phi_0}{h}} a^w_{\Gamma}(x,hD_x)\circ b^w_{\Gamma}(x,hD_x) e^{\frac{\Phi_0}{h}} \\
= \left(\frac{{\rm det}\,(\Phi''_{0,\overline{x}x})}{(\pi h)^n}\right)^2
\int\!\!\!\int a_{\Phi_0}(y) b_{\Phi_0}(z)\, 2^{4n}\, \Sigma_{y}\circ \Sigma_{z}\,L(dy)\,L(dz) \\
= \left(\frac{{\rm det}\,(\Phi''_{0,\overline{x}x})}{(\pi h)^n}\right)^3 \int\!\!\!\int\!\!\!\int
a_{\Phi_0}(y) b_{\Phi_0}(z)\, 2^{4n}\, e^{\frac{8i}{h}{\rm Im}\, \Psi_0\left(x-y,\overline{x}- \overline{z}\right)} 2^{2n} \Sigma_{x}\,L(dx)\, L(dy)\, L(dz),
\end{multline}
and therefore the operator $c^w(x,hD_x) = a^w(x,hD_x)\circ b^w(x,hD_x)$ satisfies
\begeq
\label{comp.14.3}
e^{-\frac{\Phi_0}{h}} c^w_{\Gamma}(x,hD_x) e^{\frac{\Phi_0}{h}}  = \frac{{\rm det}\,(\Phi''_{0,\overline{x}x})}{(\pi h)^n}
\int c_{\Phi_0}(x) 2^{2n}\Sigma_{x}\,L(dx),
\endeq
where
\begin{multline}
\label{comp.14.4}
c_{\Phi_0}(x) = \left(\frac{{\rm det}\,(\Phi''_{0,\overline{x}x})}{(\pi h)^n}\right)^2 \int\!\!\!\int
a_{\Phi_0}(y) b_{\Phi_0}(z)\, 2^{4n}\, e^{\frac{8i}{h}{\rm Im}\, \Psi_0\left(x-y,\overline{x}- \overline{z}\right)}\, L(dy)\, L(dz) \\
= \left(\frac{{\rm det}\,(\Phi''_{0,\overline{x}x})}{(\pi h)^n}\right)^2 \int\!\!\!\int
a_{\Phi_0}(x + y)\, b_{\Phi_0}(x + z)\, 2^{4n}\, e^{\frac{8i}{h}{\rm Im}\, \Psi_0\left(y,\overline{z}\right)}\, L(dy)\, L(dz).
\end{multline}
Let us rewrite (\ref{comp.14.4}) in more invariant terms. When doing so, we make the following two observations.
\begin{itemize}
\item[(i)] The restriction of the complex symplectic $(2,0)$--form $\sigma$ on $\comp^{2n}$ to $\Lambda_{\Phi_0}$ is given by
$$
\sigma(Y,Z) = -4{\rm Im}\, \left(\Phi''_{0,\overline{x}x}y\cdot \overline{z}\right) = -4 {\rm Im}\, \Psi_0(y,\overline{z}),
$$
where $Y,Z\in \Lambda_{\Phi_0}$ are the points in $\Lambda_{\Phi_0}$ above $y,z\in \comp^n$, respectively.
\item[(ii)] The symplectic volume form on $\Lambda_{\Phi_0}$, $\displaystyle \frac{\sigma^n}{n!}|_{\Lambda_{\Phi_0}}$, is equal to
$$
dX = 2^{2n} {\rm det}\,(\Phi''_{0,\overline{x}x})\, L(dx),\quad X = \left(x,\frac{2}{i}\frac{\partial \Phi_0}{\partial x}(x)\right)\in \Lambda_{\Phi_0},
$$
see also (\ref{comp.03}) and the following comment. 
\end{itemize}
We get therefore from (\ref{comp.14.4}),
\begeq
\label{comp14.6}
c(X) = (a\#b)(X) =  \frac{1}{(\pi h)^{2n}} \int\!\!\!\int_{\Lambda_{\Phi_0} \times \Lambda_{\Phi_0}} e^{-2i\sigma(Y,Z)/h} a(X+Y)b(X+Z)\, dY\,dZ.
\endeq

\medskip
\noindent
{\it Remark}. The integral representation formula (\ref{comp14.6}) can also be obtained directly from the corresponding formula in the real domain~\cite{Ho85},~\cite[Chapter 4]{Zw12}, thanks to the metaplectic invariance of the Weyl calculus~\cite{Sj96},~\cite{HiSj15}.

\bigskip
\noindent
We would next like to rewrite the expression (\ref{comp.14.4}) for $c_{\Phi_0}$ in terms of a suitable Gaussian Fourier multiplier on $\comp^{2n}$, acting on $a_{\Phi_0}\otimes b_{\Phi_0}$, similarly to the Weyl composition formula in the real domain~\cite{DiSj99}. To this end, introducing the positive definite Hermitian matrix $B = \Phi''_{0,\overline{x}x}$ and performing the change of variables
$$
Y = 2B^{1/2}y,\quad Z = 2B^{1/2}z,
$$
in (\ref{comp.14.4}), we obtain
\begin{multline}
\label{comp.15}
c_{\Phi_0}(x) = \frac{1}{(\pi h)^{2n}} \int\!\!\!\int a_{\Phi_0}\left(x + \frac{1}{2}B^{-1/2}Y\right)\,
b_{\Phi_0}\left(x + \frac{1}{2}B^{-1/2}Z\right) \, e^{\frac{2i}{h}{\rm Im}\,(Y\cdot \overline{Z})}\, L(dY)\, L(dZ) \\
=\frac{1}{(\pi h)^{2n}} \int\!\!\!\int a_{\Phi_0}\left(x + \frac{1}{2}B^{-1/2}Y\right)\,
b_{\Phi_0}\left(x + \frac{1}{2}B^{-1/2}Z\right) \, e^{-\frac{2i}{h}\sigma_{{\mathbf R}}(Z,Y)}\, L(dY)\, L(dZ).
\end{multline}
Here we have noticed that
\begeq
\label{comp.15.1}
{\rm Im}\,(Y\cdot \overline{Z}) = \sigma_{{\mathbf R}}(Y,Z) = -\sigma_{{\mathbf R}}(Z,Y),
\endeq
where $\sigma_{{\mathbf R}}$ is the standard symplectic form on $\real^{2n}$, when identifying this space with $\comp^n$ with the help of the map $\comp^n \ni Y = y + i\eta \mapsto (y,\eta)\in \real^{2n}$. Recall next that if $A$ is an $N\times N$ real symmetric non-degenerate matrix, we have for $u\in {\cal S}(\real^N)$,
\begeq
\label{comp.15.2}
e^{\frac{ih}{2} AD\cdot D} u(x) = \frac{1}{(2\pi h)^{N/2}} \frac{e^{\frac{i\pi}{4}{\rm sgn} A}}{\abs{{\rm det}\, A}^{1/2}} \int e^{-\frac{i}{2h}A^{-1}y\cdot y} u(x+y)\, dy.
\endeq
Applying (\ref{comp.15.2}) with $\real^{N} =  \real^{2n}_{z,\zeta} \times \real^{2n}_{y,\eta}$, and
$$
AD\cdot D = \sigma_{{\mathbf R}}(D_z,D_{\zeta};D_y,D_{\eta}) = D_{\zeta}\cdot D_y - D_{\eta}\cdot D_z,
$$
we get using the complex notation $Y = y + i\eta$, $Z = z + i\zeta$,
\begeq
\label{comp.15.3}
\left(e^{\frac{ih}{2} \sigma_{{\mathbf R}}(D_z,D_{\zeta};D_y,D_{\eta})} u(Z,Y)\right)|_{Z = Y = 0} = \frac{1}{(\pi h)^{2n}} \int\!\!\!\int e^{-\frac{2i}{h} \sigma_{{\mathbf R}}(Z,Y)} u(Z,Y)\, L(dY)\, L(dZ).
\endeq
Here we also have
\begeq
\label{comp15.4}
\sigma_{{\mathbf R}}(D_z,D_{\zeta};D_y,D_{\eta}) = \frac{2}{i}\left(D_{\overline{Z}}\cdot D_Y - D_Z \cdot D_{\overline{Y}}\right) =  \frac{2}{i} \sigma_{{\mathbf R}}(D_Z,D_{\overline{Z}};D_Y,D_{\overline{Y}}),
\endeq
where
$$
D_Y = \frac{1}{2} \left(D_y -iD_{\eta}\right), \quad D_{\overline{Y}} = \frac{1}{2} \left(D_y + iD_{\eta}\right),
$$
with $D_Z$, $D_{\overline{Z}}$ being defined similarly, see also (\ref{uni.1.1}). Combining (\ref{comp.15}), (\ref{comp.15.3}), and (\ref{comp15.4}), we get
\begin{multline}
\label{comp.16}
c_{\Phi_0}(x) = e^{h\sigma_{{\mathbf R}}(D_Z,D_{\overline{Z}};D_Y,D_{\overline{Y}})} \left(a_{\Phi_0}\left(x + \frac{1}{2}B^{-1/2}Y\right)\,
b_{\Phi_0}\left(x + \frac{1}{2}B^{-1/2}Z\right)\right)|_{Y = Z = 0} \\
=\exp\left(\frac{ih}{2}\frac{\left({}^tB^{-1} D_x\cdot D_{\overline{y}} - {}^tB^{-1} D_y\cdot D_{\overline{x}}\right)}{2i}\right)\left(a_{\Phi_0}(x)b_{\Phi_0}(y)\right)|_{y=x}.
\end{multline}
Here the symbol of the second order constant coefficient differential operator on $\comp^{2n}_{x,y}$,
$$
\frac{1}{2i}\left({}^tB^{-1} D_x\cdot D_{\overline{y}} - {}^tB^{-1} D_y\cdot D_{\overline{x}}\right)
$$
is a quadratic form on $\comp^{2n}_{\xi,\eta}$ given by
\begeq
\label{comp.16.1}
\frac{1}{8i}\left({}^tB^{-1}\overline{\xi}\cdot \eta - {}^tB^{-1} \overline{\eta}\cdot \xi\right) = -\frac{1}{4} {\rm Im}\, \left(B^{-1}\xi \cdot \overline{\eta}\right).
\endeq
Letting
\begeq
\label{comp.16.2}
\sigma_{\Phi_0} = \frac{2}{i} \sum_{j,k=1}^n \frac{\partial^2 \Phi_0}{\partial \overline{x}_j \partial x_k} d\overline{x}_j \wedge dx_k
\endeq
be the pullback of the complex symplectic form $\sigma$ on $\comp^{2n}$ under the map
$$
\comp^n \ni x \mapsto \left(x,\frac{2}{i}\frac{\partial \Phi_0}{\partial x}(x)\right)\in \Lambda_{\Phi_0} \subset \comp^{2n},
$$
we see that
$$
\sigma_{\Phi_0}(\xi,\eta) = -4 {\rm Im}\left(B\xi \cdot \overline{\eta}\right),\quad \xi,\eta \in \comp^n,
$$
and therefore the quadratic form in (\ref{comp.16.1}) can be regarded as the dual to $\sigma_{\Phi_0}$, when the latter is viewed as a quadratic form on $\comp^{2n}_{\xi,\eta}$. Setting
$$
\sigma^{-1}_{\Phi_0}(\xi,\eta)= -\frac{1}{4}{\rm Im}\, \left(B^{-1}\xi \cdot \overline{\eta}\right),
$$
we may summarize the discussion above in the following result.
\begin{prop}
\label{comp0}
Let $a,b\in {\cal S}(\Lambda_{\Phi_0})$ and let $c^w(x,hD_x) = a^w(x,hD_x)\circ b^w(x,hD_x)$. The symbol $c\in {\cal S}(\Lambda_{\Phi_0})$ is given by
\begeq
\label{comp.17}
c(X) = (a\#b)(X) =  \frac{1}{(\pi h)^{2n}} \int\!\!\!\int_{\Lambda_{\Phi_0} \times \Lambda_{\Phi_0}} e^{-2i\sigma(Y,Z)/h} a(X+Y)b(X+Z)\, dY\,dZ.
\endeq
We also have
\begin{equation}
\label{comp.18}
c_{\Phi_0}(x) = \exp\left(\frac{ih}{2}\sigma_{\Phi_0}^{-1}(D_{x,\overline{x}},D_{y,\overline{y}})\right)(a_{\Phi_0}(x)b_{\Phi_0}(y))_{\big|y=x}.
\end{equation}
\end{prop}

\medskip
\noindent
{\it Remark}. We refer to the recent work~\cite{HiLaSjZe20} for an alternative approach to the composition formulas for the semiclassical Weyl calculus in the complex domain, based on the Fourier inversion formula on $\Lambda_{\Phi_0}$ and the method of magnetic translations.

\bigskip
\noindent
We shall finish this subsection by discussing the composition formula (\ref{comp.17}) in the case when $a,b\in {\cal G}^s_b(\Lambda_{\Phi_0})$, for some $s>1$. It has been established in~\cite{La88},~\cite{LaLa97}, working in the real domain, that we then have
$c = a\#b \in {\cal G}^s_b(\Lambda_{\Phi_0})$. The argument in~\cite{La88} proceeds by repeated partial integrations and suitable quasinorm estimates, and our purpose here is to provide an alternative approach to the proof of this result, making use of the method of contour deformations. When doing so, rather than working with Gevrey symbols of $\Lambda_{\Phi_0}$, in view of the metaplectic invariance of the Weyl calculus~\cite{Sj96},~\cite{HiSj15}, it will be sufficient for us to work on $\real^m \simeq T^*\real^n$, where $m=2n$.

\medskip
\noindent
Let $a,b\in {\cal G}^s_b(\real^m)$ and let us set following (\ref{comp.17}),
\begeq
\label{comp.18.1}
c(X) = (a\#b)(X) =  \frac{1}{(\pi h)^m} \int\!\!\!\int_{{\bf R}^m \times {\bf R}^m} e^{-2i\sigma(Y,Z)/h} a(X+Y)b(X+Z)\, dY\,dZ.
\endeq
Here $\sigma$ is the standard symplectic form on $\real^m$, and the integral in (\ref{comp.18.1}) is an oscillatory one. Let $\chi\in {\cal G}^s_0(\real^{2m})$ be such that $\chi(Y,Z) = 1$ for $\abs{(Y,Z)}\leq 1$, with ${\rm supp}\, \chi \subset B(0,2)$, and define also
\begeq
\label{comp.18.2}
r_{\chi}(X)=\frac{1}{(\pi h)^m} \int\!\!\!\int_{{\bf R}^m \times {\bf R}^m} e^{-2i\sigma(Y,Z)/h} \left(1-\chi(Y,Z)\right)a(X+Y)b(X+Z)\, dY\,dZ.
\end{equation}

\noindent
The standard semiclassical calculus~\cite{DiSj99} gives that $\norm{\partial^{\alpha}r_{\chi}}_{L^{\infty}({\bf R}^m)} =
{\cal O}_{\alpha}(h^{\infty})$, for all $\alpha \in \nat^m$, and we would like to sharpen these asymptotic bounds, thanks to the Gevrey smoothness of the symbols $a,b$. To this end, we have the following result, due to~\cite{La88},~\cite{LaLa97}.
\begin{prop}
\label{Gevrey_remainder}
Let $a,b\in {\cal G}^s_b(\real^m)$, for some $s>1$, and let us define $r_{\chi}\in C^{\infty}(\real^m)$ as in {\rm (\ref{comp.18.2})}. There exists $C>0$ such that for all $\alpha\in \nat^m$ and $h\in (0,1]$, we have
\begin{equation}
\label{comp.18.3}
\abs{\partial^\alpha r_\chi(X)} \leq C^{1+|\alpha|}(\alpha!)^s\,\exp\left(-\frac{1}{\mathcal{O}(1)}h^{-\frac{1}{s}}\right),\quad X\in \real^m,
\end{equation}
\end{prop}
\begin{proof}
We shall prove the following more general statement, implying (\ref{comp.18.3}): let $q(x)$ be a real valued non-degenerate quadratic form on $\real^N$, let $a\in {\cal G}^s_b(\real^N)$, for some $s>1$, and let $\chi \in {\cal G}^s_0(\real^N)$ be such that $\chi(x) = 1$ for $\abs{x}\leq 1$, ${\rm supp}\, \chi \subset B_{{\bf R}^N}(0,2)$. Setting
\begeq
\label{comp.18.4}
r_{\chi}(x) = h^{-N/2} \int e^{iq(y)/h} (1-\chi(y)) a(x+y)\, dy,
\endeq
we shall prove that $r_{\chi}$ enjoys the same estimates as in (\ref{comp.18.3}). To this end, let $\widetilde{a}\in {\cal G}^s_b(\comp^N)$ be an almost holomorphic extension of $a$ such that ${\rm supp}\, \widetilde{a} \subset \real^N + iB_{{\bf R}^N}(0,C)$, for some $C>0$, and let $\widetilde{\chi}\in {\cal G}^s_0(\comp^N)$ be an almost holomorphic extension of $\chi$, with ${\rm supp}\, \widetilde{\chi}$ close to that of $\chi$. We shall replace the integration in (\ref{comp.18.4}) along $\real^N$ by the integration along the contour
\begeq
\label{comp.18.5}
\Gamma_{\theta_0}:\quad \real^N \ni y\mapsto y + i\theta_0 \frac{q'(y)}{\abs{q'(y)}}\in \comp^N,\quad \abs{y}\geq \frac{1}{2},
\endeq
for some $\theta_0 > 0$ small enough, where we notice that along $\Gamma_{\theta_0}$, we have
\begeq
\label{comp.18.6}
{\rm Im}\, q(z) = \theta_0 \abs{q'(y)} \asymp \theta_0 \abs{y}, \quad z\in \Gamma_{\theta_0},
\endeq
since $q$ is non-degenerate. Introducing also the damping factor $e^{-\varepsilon y^2/2}$, $\varepsilon >0$, in (\ref{comp.18.4}), we get by an application of Stokes formula,
\begin{multline}
\label{comp.19}
\int_{{\bf R}^N} e^{iq(y)/h} e^{-\varepsilon y^2/2} (1-\chi(y)) a(x+y)\, dy = \int_{\Gamma_{\theta_0}} e^{iq(z)/h} e^{-\varepsilon z^2/2} (1-\widetilde{\chi}(z)) \widetilde{a}(x+z)\, dz \\
+ \int\!\!\!\int_{G_{[0,\theta_0]}} e^{iq(z)/h} e^{-\varepsilon z^2/2} \overline{\partial} \left((1-\widetilde{\chi}(z)) \widetilde{a}(x+z)\right)\wedge\, dz.
\end{multline}
Here $G_{[0,\theta_0]}\subset \comp^N$ is the $(n+1)$--dimensional contour given by
$$
G_{[0,\theta_0]} = \bigcup_{\theta\in[0,\theta_0]}\Gamma_\theta,
$$
with $\Gamma_{\theta}$ defined similarly to (\ref{comp.18.5}). Taking $\displaystyle \theta_0 = \frac{1}{C_0} h^{1-\frac{1}{s}}$, for some constant $C_0 > 0$ large enough, we obtain in view of (\ref{comp.18.6}),
\begeq
\label{comp.20}
\int_{\Gamma_{\theta_0}} e^{iq(z)/h} e^{-\varepsilon z^2/2} (1-\widetilde{\chi}(z)) \widetilde{a}(x+z)\, dz = {\cal O}(1)\, \exp\left(-\frac{1}{{\cal O}(1)}h^{-\frac{1}{s}}\right),
\endeq
uniformly in $\varepsilon >0$. Furthermore, using (\ref{ahol.3}) and (\ref{comp.18.6}), we see that for some $C_1>0$, the second term in the right hand side of (\ref{comp.19}) is of the form
\begin{multline}
\label{comp.21}
{\cal O}(1) \int_0^{\theta_0} d\theta \int_{\abs{y}\geq 1/2} e^{-\frac{\theta \abs{y}}{{\cal O}(1)h}} \exp\left(-\frac{1}{C_1} \theta^{-\frac{1}{s-1}}\right)\, dy \\ \leq {\cal O}(1) \int_0^{\theta_0} \theta^{-N} \exp\left(-\frac{1}{C_1} \theta^{-\frac{1}{s-1}}\right)\, d\theta
\leq {\cal O}(1)\,\exp\left(-\frac{1}{{\cal O}(1)}h^{-\frac{1}{s}}\right),
\end{multline}
uniformly in $\varepsilon >0$. Here we have also used the fact that the function
$$
t\rightarrow \exp\left( -\frac{1}{2C_1} t^{-\frac{1}{s-1}}\right)
$$
is increasing on $[0,\theta_0]$. Combining (\ref{comp.19}), (\ref{comp.20}), and (\ref{comp.21}) and letting $\varepsilon \rightarrow 0^+$, we get that the oscillatory integral in (\ref{comp.18.4}) satisfies,
\begeq
\label{comp.22}
r_{\chi}(x) = {\cal O}(1)\,\exp\left(-\frac{1}{\mathcal{O}(1)}h^{-\frac{1}{s}}\right),\quad x\in \real^N.
\endeq
Considering the derivatives of $r_{\chi}$ in (\ref{comp.18.4}) and using the fact that for each $\alpha \in \nat^N$, the function $\partial^{\alpha}_x \widetilde{a} \in {\cal G}^s_b(\comp^N)$ is an almost holomorphic extension of $\partial^{\alpha}a\in {\cal G}^s_b(\real^N)$, we obtain, arguing as above,
\begeq
\label{comp.23}
\abs{\partial^\alpha r_\chi(x)} \leq C^{1+|\alpha|}(\alpha!)^s\,\exp\left(-\frac{1}{\mathcal{O}(1)}h^{-\frac{1}{s}}\right),\quad x\in \real^N.
\endeq
The proof is complete.
\end{proof}

\bigskip
\noindent
Continuing to use the notation in the proof of Proposition \ref{Gevrey_remainder}, let us also consider
\begeq
\label{comp.24}
\ell_{\chi}(x) = h^{-N/2} \int e^{iq(y)/h} \chi(y) a(x+y)\, dy,
\endeq
where we write $\displaystyle q(y) = \frac{1}{2} Ay\cdot y$. Letting
$$
C_A = \frac{(2\pi)^{N/2} e^{i\frac{\pi}{4} \textrm{sgn}(A)}}{\abs{{\rm det}\,A}^{1/2}},
$$
where $\textrm{sgn}(A)$ is the signature of $A$, and
$$
\ell_{\chi,K}(x) = \ell_{\chi}(x) - C_A \sum_{k=0}^{K-1} \frac{h^k}{k!} \left(\frac{1}{2i} A^{-1}D\cdot D\right)^k a(x),\quad K = 1,2,\ldots,
$$
we conclude by quadratic stationary phase and the fact that $\chi_0\in {\cal G}^s_0(\real^N)$, $a\in {\cal G}^s_b(\real^N)$, that there exists $C>0$ such that for all $\alpha \in \nat^N$, $K\in \nat$, we have
\begeq
\label{comp.25}
\abs{\partial^{\alpha}_x \ell_{\chi,K}(x)} \leq C^{1+K+\abs{\alpha}} (K!)^{2s-1} (\alpha!)^s h^K.
\endeq
In particular, we have $\ell_{\chi}\in {\cal G}^s_b(\real^N)$, and here once again we encounter the phenomenon of the loss of Gevrey smoothness in stationary phase expansions, see also (\ref{psg.11.01}), (\ref{psg.11.02}).

\medskip
\noindent
The discussion above gives, in particular, an alternative proof of the following result due to~\cite{La88},~\cite{LaLa97}.
\begin{coro}
Let $a\in {\cal G}^s_b(\Lambda_{\Phi_0})$, $b\in {\cal G}^s_b(\Lambda_{\Phi_0})$, for some $s>1$. Then the symbol $c = a\#b$, defined in {\rm (\ref{comp.17})}, satisfies $c\in {\cal G}^s_b(\Lambda_{\Phi_0})$.
\end{coro}

\end{document}